\documentclass[letterpaper,12pt,oneside,final]{article}

\usepackage{cite}

\usepackage{ifthen}
\newboolean{ElectronicVersion}
\setboolean{ElectronicVersion}{true} 

\usepackage{amsmath,amssymb,amstext} 
\usepackage[pdftex]{graphicx} 
\usepackage{setspace}
\usepackage{subfig}

\usepackage{epstopdf} 
\usepackage{hyperref}

\usepackage[top=1in, bottom=1in, left=1in, right=1in]{geometry} 

\usepackage{algorithm, algorithmicx}
\usepackage[noend]{algpseudocode}
\algblockdefx[Step]{Step}{EndStep}
    [2]{{\bf Step #1:} #2}
    {}
\algtext*{Procedure}
\algtext*{EndWhile}
\algtext*{EndIf}   
\algtext*{EndStep}
\algtext*{EndFor}

\newtheorem{theorem}{Theorem}[section]

\newtheorem{lemma}{Lemma}[section]

\newcommand{\R}{\mathbb{R}}

\newcommand{\E}{\mathbb{E}}

\newcommand{\diag}{\mathrm{diag}\,}
\newcommand{\Diag}{\mathrm{Diag}\,}

\newcommand{\rank}{\mathrm{rank}\,}
\newcommand{\sign}{\mathrm{sign}\,}

\newcommand{\supp}{\mathrm{supp}\,}
\newcommand{\st}{\mathrm{s.t.}\;}
\newcommand{\tr}{\mathrm{Tr}\,}

\newcommand{\e}{\mathbf{e}}

\newcommand{\n}{\mathbf{n}}

\renewcommand{\u}{\mathbf{u}}
\newcommand{\bv}{\mathbf{v}}
\newcommand{\x}{\mathbf{x}}
\newcommand{\y}{\mathbf{y}}

\newcommand{\id}{\mathbf{1}}

\newcommand{\abs}[1]{\ensuremath{\left| #1  \right| }}
\newcommand{\mat}[1] {\ensuremath{ \left(\begin{array} #1 \end{array} \right)}} 
\newcommand{\branchdef}[1] {\ensuremath{ \left\{\begin{array}{rl} #1 \end{array} \right. }} 
\newcommand{\inp}[1] {\ensuremath{\langle #1 \rangle}} 
\newcommand{\bra}[1] {\ensuremath{ \left\{ #1\right \}}} 
\newcommand{\bbra}[1]{\ensuremath{ \big( #1 \big) } } 
\newcommand{\Bbra}[1]{\ensuremath{ \Big( #1 \Big) } } 
\newcommand{\sbra}[1] {\ensuremath{ \left[ #1\right]}} 
\newcommand{\rbra}[1]{\ensuremath{\left( #1 \right)}} 

\newcommand{\ds}[1]{\displaystyle{#1}}

\newcommand{\whp}{with high probability}

\newcommand{\barX}{\bar\bv \bar\bv^T}
\newcommand{\bubv}{\u \bv^T}
\newcommand{\kbs}{$(k_1,k_2)$-subgraph}

\newcommand{\ra}{\rightarrow}

\newcommand{\qed}{\hfill\rule{2.1mm}{2.1mm}}

\begin{document}


\title{Guaranteed Recovery of Planted Cliques and Dense Subgraphs by Convex Relaxation
\author{Brendan P.W.~Ames}
\thanks{University of Alabama,
Department of Mathematics, Box 870350, Tuscaloosa, AL, 35487-0350 +1-205-348-5155
bpames@ua.edu }
}

\maketitle

\begin{abstract}
	We consider the problem of identifying the densest $k$-node subgraph in a given graph.
We write this problem as an instance of rank-constrained cardinality minimization and
then relax using the nuclear and $\ell_1$ norms.
Although the original combinatorial problem is NP-hard, we show that the densest $k$-subgraph can be recovered
from the solution of our convex relaxation for certain program inputs.
In particular, we establish exact recovery in the case that the input graph contains a single planted clique plus noise
in the form of corrupted adjacency relationships.
We consider two constructions for this noise. In the first, noise is introduced by an adversary deterministically deleting edges
within the planted clique and placing diversionary edges.
In the second, these edge corruptions are performed at random.
Analogous recovery guarantees for identifying the densest subgraph of fixed size
in a bipartite graph are also established, and
results of numerical simulations for randomly generated graphs are included to demonstrate the efficacy of our algorithm.
	
\end{abstract}	


\allowdisplaybreaks

\section{Introduction}
%
We consider the {\it densest $k$-subgraph problem}. Given input graph $G$ and integer $k$, 
the densest $k$-subgraph problem seeks the $k$-node subgraph of $G$ with maximum number of edges.
The identification and analysis of dense subgraphs plays a significant role	
in a wide range of applications, including information retrieval, pattern recognition, computational biology, and image processing.
For example, a group of densely connected nodes may correspond to a community of users in a social network
or cluster of similar items in a given data set.
Unfortunately, the problem of finding a densest subgraph of given size is known to be both NP-hard (see \cite{feige2001dense}) and  hard to approximate (see \cite{feige2002relations,khot2004ruling,alon2011inapproximability}).

Our results can be thought of as a generalization to the densest $k$-subgraph problem of those in \cite{ames2011nuclear} for the maximum clique problem.
In \cite{ames2011nuclear}, Ames and Vavasis establish that the maximum clique of a given graph can be recovered from the optimal solution of a particular convex program for certain classes of input graphs. 
Specifically, Ames and Vavasis show that the maximum clique in a graph consisting of a single large clique, called a planted clique, and a moderate amount of diversionary edges and nodes
 can be identified from the minimum nuclear norm solution of a particular system of linear inequalities. 
These linear constraints restrict all feasible solutions to be adjacency matrices of subgraphs with a desired number of nodes, say $k$, while the objective acts as a surrogate for the rank of the feasible solution;
a rank-one solution would correspond to a $k$-clique in the input graph.
We establish analogous recovery guarantees for a convex relaxation of the planted clique problem that is robust to noise in the form of both diversionary edge additions and {\it deletions} within the planted complete subgraph.
In particular, we modify the relaxation of \cite{ames2011nuclear} by adding an $\ell_1$ norm penalty to measure the error between the rank-one approximation of the adjacency matrix of each $k$-subgraph and its true adjacency matrix.

This relaxation technique, and its  accompanying recovery guarantee, mirrors that of several recent papers regarding convex optimization approaches for robust principal component analysis  \cite{chandrasekaran2009rank,candes2009robust,chen2011low}
and graph clustering \cite{oymak2011finding,jalali2011clustering,chen2012clustering}. 
Each of these papers establishes that a desired matrix or graph structure, represented as the sum of a low-rank and sparse matrix,
can be recovered from the optimal solution of some convex program under certain assumptions on the input matrix or graph.
In particular, our analysis and results are closely related to those of \cite{chen2012clustering}.
In \cite{chen2012clustering}, Chen et al.~consider a convex optimization heuristic for identifying clusters in data, represented as collections of relatively dense subgraphs in a sparse graph, and provide
bounds on the size and density of these subgraphs ensuring exact recovery using this method.
We establish analogous guarantees for identifying a single dense subgraph when the cardinality of this subgraph is known a priori.
For example, we will show that a planted clique of cardinality as small as $\Omega(N^{1/3})$  can be recovered in the presence of sparse random noise, where $N$ is the number of nodes in the input graph, significantly less than the
bound $\Omega(N^{1/2})$ established in \cite{ames2011nuclear}.

The remainder of the paper is structured as follows.
We present our relaxation for the densest $k$-subgraph problem and state our theoretical recovery guarantees
in Section~\ref{sec: dks prob}.
In particular, we will show that the densest $k$-subgraph can be recovered from the optimal solution of our convex relaxation
in the case that the input graph $G=(V,E)$ consists of a planted $k$-clique $V^*$ that has been corrupted by the noise in the form of diversionary edge additions and deletions, as well as diversionary nodes.
We consider two cases. In the first, noise is introduced deterministically by an adversary adding diversionary edges and deleting edges between nodes within the planted clique.
In the second, these edge deletions and additions are performed at random.
We present an analogous relaxation for identifying the densest bipartite subgraph of given size in a bipartite graph in Section~\ref{sec: dkb prob}.
A proof of the recovery guarantee for the densest $k$-subgraph problem in the random noise case comprises Section~\ref{sec: proof}; the proofs of the remaining theoretical guarantees are similar and 
are included as supplemental material.
We conclude with simulation results for synthetic data sets in Section~\ref{sec: expts}.

\section{The Densest $k$-Subgraph Problem}
\label{sec: dks prob}
The \emph{density} of a graph $G = (V, E)$ is defined to be the average number of edges incident at a vertex or average degree of $G$:
$
	d(G) = |E|/|V|. 
$
The \emph{densest $k$-subgraph problem} seeks a $k$-node subgraph of $G$ 
of maximum average degree or density:
\begin{equation}	\label{eq: dks problem}
	\max \{d(H): H \subseteq G, \, |V(H)| = k \}.
\end{equation}
Although the problem of finding a subgraph with maximum average degree is polynomially solvable (see \cite[Chapter 4]{lawlercombinatorial}),
the densest $k$-subgraph problem is NP-hard. Indeed, 
if a graph $G$ has a clique of size $k$, this clique would be the densest $k$-subgraph of $G$.
Thus, any instance of the maximum clique problem, known to be NP-hard \cite{karp1972reducibility},
is equivalent to an instance of the densest $k$-subgraph problem.
Moreover, the densest $k$-subgraph problem is hard to approximate;
specifically, it has been shown that the densest $k$-subgraph problem does not admit a polynomial-time approximation scheme under various complexity theoretic assumptions \cite{feige2002relations,khot2004ruling,alon2011inapproximability}.
Due to, and in spite of, this intractability of the densest $k$-subgraph problem, we consider relaxation of \eqref{eq: dks problem} to a convex program.
Although we do not expect this relaxation to provide a good approximation of the densest $k$-subgraph for every input graph, we will
establish that the densest $k$-subgraph can be recovered from the optimal solution of this convex relaxation for certain classes of input graphs.
In particular, we will show that our relaxation is exact for graphs containing a single dense subgraph obscured by noise in the form of diversionary nodes and edges.

Our relaxation is based on
the observation that the adjacency matrix of a dense subgraph is well-approximated by the rank-one adjacency matrix of the complete graph on the same node set.
Let $V' \subseteq V$ be a subset of $k$ nodes of the graph $G = (V,E)$ and
let $\bar \bv$ be its characteristic vector. That is, for all $i \in V$, $\bar \bv_i  = 1$ if $i \in V'$ and is equal to $0$ otherwise.
The vector $\bar\bv$ defines a rank-one matrix $\bar X$ by the outer product of $\bar\bv$ with itself: $\bar X = \bar\bv \bar\bv^T$.
Moreover, if $V'$ is a clique of $G$ then the nonzero entries of $\bar X$ correspond to the $k\times k$ all-ones block of the 
perturbed adjacency matrix $\tilde A_G := A_G + I$ of $G$ indexed by $V'  \times V'$;
here $A_G \in \R^{V\times V}$ denotes the \emph{adjacency matrix} of the graph $G$, defined by
$$
	[A_G]_{ij} := \branchdef{ 1, & \mbox{if $ij \in E$,} \\ 0, &\mbox{otherwise,} }
$$
and $I \in \R^{V\times V}$ denotes the identity matrix with rows and columns indexed by $V$.
If $V'$ is not a clique of $G$, then the entries of $\tilde A_G (V', V')$ indexed by nonadjacent nodes are
equal to $0$. Let $\bar Y \in \R^{V\times V}$ be the matrix defined by
\begin{equation}	\label{eq: proposed Y}
	\bar Y_{ij} := \branchdef{ 	-\bar X_{ij},  &\mbox{if $ij \in \tilde E$}, \\
							0, &\mbox{otherwise,} 
			}
\end{equation}
where $\tilde E$ is the complement of the edge-set of $G$ given by $\tilde E:= (V\times V) - E - \{uu: u \in V\}$.
That is, $\bar Y = - P_{\tilde E}(\bar X)$, where $P_{\tilde E}$ is the orthogonal projection onto the set of matrices with support contained $\tilde E$ defined by
\begin{equation} \label{eq: proj def}
	[ P_{\tilde E}(M)]_{ij} := \branchdef{ M_{ij}, &\mbox{if } (i,j) \in \tilde E, \\ 0, &\mbox{otherwise,} } 
\end{equation}
for all $M \in \R^{V\times V}$.
The matrix $\bar Y$ can be thought of as a correction for the entries of $\bar X$ indexed by nonedges of $G$.
Indeed, $\bar X + \bar Y$ is exactly the adjacency matrix of the subgraph of $G$ induced by $V'$, with
ones in diagonal entries indicating loops at each $v \in V'$:
$$
 	\bar X_{ij} + \bar Y_{ij} = \branchdef{ 	1, & \mbox{if  } ij \in (E \cup \{ uu: u \in V\}) \cap (V' \times V'), \\
 									0, &\mbox{otherwise.} 	
 								}
$$ 									
 Moreover,   the density of $G(V')$ is equal to
$$
	d(G(V')) = \frac{1}{2k} \bbra{ k(k-1) - \| \bar Y\|_0 },
$$
by the fact that the number of nonzero entries in $\bar Y$ is exactly twice the number of
 nonadjacent pairs of nodes in $G(V')$,
Here $\|\bar Y \|_0$ denotes the so-called {\it $\ell_0$ norm} of $\bar Y$, defined as the cardinality of the support of $\bar Y$.
Maximizing the density of $H$ over all $k$-node subgraphs of $G$ is equivalent to 
minimizing $\|Y\|_0$ over all $(X, Y)$ as constructed above. 
Consequently, \eqref{eq: dks problem} is equivalent to
\begin{equation}	\label{eq: dks min l0}
	\min_{X, Y\in\Sigma^V} \left\{ \|Y\|_0: \rank (X) = 1, \; \e^T X \e = k^2, \; X_{ij} + Y_{ij} = 0  \; \forall \, ij \in \tilde E, \;X  \in \{0,1\}^{V \times V} \right\}, 
\end{equation}
where $\e$ is the all-ones vector in $\R^V$,
$\Sigma^V$ denotes the cone of $|V|\times |V|$ symmetric matrices with rows and columns indexed by $V$.
Indeed, the constraints $\rank(X) = 1,$ $\e^T  X \e = k^2,$ and $X \in \Sigma^V \cap \{0,1\}^{V\times V}$ force any feasible $X$ to be a rank-one symmetric binary matrix with exactly $k^2$ nonzero entries,
while the requirement that $X_{ij} + Y_{ij} = 0$ if $ij \in \tilde E$ ensures that every entry of $X$ indexed by a nonadjacent pair of nodes is corrected by $Y$.
Moving the constraint $\rank(X) = 1$ to the objective as a penalty term yields the nonconvex program
\begin{equation}	\label{eq: dks RPCA}
	\min_{X, Y \in \Sigma^V} \left\{ \rank(X) + \gamma\|Y\|_0: \; \e^T X \e = k^2, \; X_{ij} + Y_{ij} = 0  \; \forall \, ij \in \tilde E, \;X \in \{0,1\}^{V\times V} \right\}.
\end{equation}
Here $\gamma > 0$ is a regularization parameter to be chosen later.
We relax \eqref{eq: dks RPCA} to the convex problem
\begin{equation} 	\label{eq: dks relaxation}
	\min \left\{ \|X\|_*+ \gamma \|Y\|_1 : \; \e^T X \e = k^2, \; X_{ij} + Y_{ij} = 0  \; \forall \, ij \in \tilde E, \; X \in [0,1]^{V\times V} \right\}
\end{equation}
by replacing $\rank$ and $\|\cdot\|_0$ with their convex envelopes, the nuclear norm $\|\cdot\|_*$ and the 
$\ell_1$ norm $\|\cdot\|_1$, relaxing the binary constraints on the entries of $X$ to the corresponding box constraints, and
omitting the symmetry constraints on $X$ and $Y$.
Although ignoring the symmetry constraints is not necessary to obtain a tractable relaxation,
our proofs of the recovery guarantees stated in Theorems~\ref{thm: dks adv guarantee}~and~\ref{thm: dks recovery} suggest that the Lagrange multipliers corresponding to the symmetry constraints in \eqref{eq: dks RPCA} may be chosen to be equal to $0$;
we omit the symmetry constraints to eliminate these $O(N^2)$ potentially redundant linear constraints and allow a simpler extension of \eqref{eq: dks relaxation} to the bipartite problem considered in Section~\ref{sec: dkb prob}.
Here $\|Y\|_1$ denotes the $\ell_1$ norm of the vectorization of $Y$:
$
	\|Y\|_1 := \sum_{i\in V} \sum_{j\in V} |Y_{ij}|.
$
Note that $\|Y\|_0 = \|Y\|_1$ for the proposed choice of $Y$ given by \eqref{eq: proposed Y}, although this equality clearly does not hold in general.

Our relaxation mirrors that proposed by Chandrasekaran et al.~\cite{chandrasekaran2009rank} for  robust principal component analysis.
Given matrix $M\in \R^{m\times n}$, the Robust PCA problem seeks a decomposition of the form $M = L + S$ where $L \in \R^{m\times n}$ has low rank and
$S \in \R^{m\times n}$ is sparse.
In \cite{chandrasekaran2009rank}, Chandrasekaran et al.~establish that such a decomposition can be obtained by solving
the convex problem
$
	\min \{ \|L\|_* + \|S\|_1: M = L+ S \}
$
under certain assumptions on the input matrix $M$.
Several recent papers \cite{candes2009robust,doan2010finding,chen2011low,oymak2011finding,jalali2011clustering,chen2012clustering} have extended this result to obtain conditions on the input matrix $M$
ensuring perfect decomposition under partial observation of $M$ and other linear constraints.
These results can be thought of as generalizations of the results of \cite{recht2010guaranteed}, which established conditions under which a low-rank matrix can recovered from  linear samples of its entries and an appropriate nuclear norm minimization; this result, in turn, generalizes those of  \cite{Gilbert,Donoho,CandesRombergTao} establishing
that the sparsest solution of some sets of linear equations can be recovered using relaxation of vector cardinality to
the vector $\ell_1$ norm.
These recovery guarantees rely on the fact that the linear sampling operators sample roughly equal amounts of information from each entry of the matrix or vector, often stated in the form of a restricted isometry or incoherence property.
Unfortunately, if the graph $G$ contains a large clique or dense subgraph then 
the linear adjacency constraints in \eqref{eq: dks relaxation} will ignore the block of $A_G$ indexed by this dense subgraph;
therefore, we can always construct graphs where \eqref{eq: dks relaxation} fails to satisfy the restricted isometry property.
Although these recovery guarantees do not translate immediately to our formulation for the densest $k$-subgraph problem and its relaxation, we will establish analogous conditions ensuring exact recovery of the densest $k$-subgraph
of $G$ from \eqref{eq: dks relaxation} under certain conditions on $G$.


We consider a planted case analysis of \eqref{eq: dks relaxation}.
Suppose that the input graph $G$ contains a single dense subgraph $H$, plus diversionary edges and nodes.
We are interested in  the tradeoff  between the density of $H$, the size $k$ of $H$, and the level of noise required to guarantee recovery of $H$ from the optimal solution of \eqref{eq: dks relaxation}.
In particular, we consider graphs $G = (V,E)$ constructed as follows.
We start by adding all edges between elements of some $k$-node subset $V^* \subseteq V$ to $E$.
That is, we create a $k$-clique $V^*$ by adding the edge set of the complete graph with vertex set $V^*$ to $G$.
We then corrupt this $k$-clique with noise in the form of deletions of edges within $V^*\times V^*$ and additions of potential edges in $(V\times V) - (V^*\times V^*)$.

We consider two cases. In the first, these additions and deletions are performed deterministically.
In the second, the adjacency of each vertex pair is corrupted independently at random.
In the absence of edge deletions, this is exactly the planted clique model considered in \cite{ames2011nuclear}.
In \cite{ames2011nuclear}, Ames and Vavasis provide conditions ensuring exact recovery of a planted clique from the optimal solution of the convex program
\begin{equation} \label{eq: AV prob}
	\min_X  \bra{ \|X \|_*: \e^T X \e \ge k^2, \; X_{ij} = 0 \; \forall \, ij \in \tilde E }.
\end{equation}	
The following theorem provides a recovery guarantee for the densest $k$-subgraph in the case of adversarial edge additions and deletions, analogous to that of \cite[Section~4.1]{ames2011nuclear}.

\begin{theorem}	\label{thm: dks adv guarantee}
	Let $V^*$ be a $k$-subset of nodes of the graph  $G = (V,E)$ and let $\bv$ be its characteristic vector.
	Suppose that $G$ contains at most $r$ edges not in $G(V^*)$ and ${G}(V^*)$ contains at least ${k \choose 2} - s$ edges, such that each vertex in $V^*$ is adjacent to at least $(1-\delta_1) k $ nodes in
	$V^*$ and each vertex in $V-V^*$ is adjacent to at most $\delta_2 k$ nodes in $V^*$
	for some $\delta_1, \delta_2 \in (0, 1)$ satisfying
	$
		2 \delta_1 + \delta_2 < 1.
	$
	Let $(X^*, Y^*) $ be the feasible solution for \eqref{eq: dks relaxation}
	where $X^* = \bv\bv^T$ and $Y^*$ is constructed according to \eqref{eq: proposed Y}.		
	Then there exist scalars $c_1, c_2 > 0$, depending only on $\delta_1$ and $\delta_2$, such that if  
	$s \le c_1 k^2$ and $r \le c_2 k^2$
	then 	$G(V^*)$ is the unique maximum density $k$-subgraph of $G$
	and $(X^*, Y^*) $ is the unique optimal solution of \eqref{eq: dks relaxation} for
	$
		\gamma =  2 \big((1 - 2 \delta_1 - \delta_2) k \big)^{-1}.
	$
\end{theorem}

In Theorem~\ref{thm: dks adv guarantee}, the constants $\delta_1$ and $c_1$ parametrize the density of the planted dense
subgraph $G(V^*)$, while $\delta_2$ and $c_2$ control the number of edges in $G$ outside of $G(V^*)$.
Specifically, $\delta_1$ denotes the minimum degree of a node in $G(V^*)$ and $\delta_2$ denotes the maximum number of
neighbours each node in $V-V^*$ may have in $V^*$. Theorem~\ref{thm: dks adv guarantee} states that if $G(V^*)$ is
sufficiently dense, then $G(V^*)$ is the densest $k$-subgraph of $G$
and can be recovered by solving the relaxation \eqref{eq: dks relaxation}; here, ``sufficiently dense" corresponds to $G(V^*)$
containing at least ${k \choose 2} - c_1 k^2$ edges and $G$ containing at most $c_2 k^2$ edges total.

The bound on the number of adversarially added edges given by Theorem~\ref{thm: dks adv guarantee} matches that given in \cite[Section~4.1]{ames2011nuclear} up to constants.
Moreover, the noise bounds given by Theorem~\ref{thm: dks adv guarantee} are optimal in the following sense.
Adding $k$ edges from any node $v'$ outside $V^*$ to $V^*$ would result in the creation of a $k$-subgraph (induced by $v'$ and $V^* - u$ for 
some $u \in V^*$) of greater density than $G(V^*)$. Similarly, if the adversary can add or delete $O(k^2)$ edges, then the adversary can create a  $k$-subgraph with greater density than $G(V^*)$.
In particular,  a $k$-clique could be created by adding at most ${k\choose 2}$ edges.

We also consider random graphs $G=(V,E)$ constructed in the following manner.
\begin{itemize}
	\item[($\rho_1$)]
		Fix subset $V^* \subseteq V$ of size $K$. 
		Add $ij$ to $E$  independently with fixed probability $1-q$ for all $(i,j)~\in~V^*~\times~V^*$.	
	\item[($\rho_2$)]
		Each of the remaining potential edges in $(V\times V) - (V^*\times V^*)$ is added independently to $E$ with fixed probability $p$.
\end{itemize}
We say such a graph $G$ is sampled from the \emph{planted dense $k$-subgraph model}.
By construction, the subgraph $G(V^*)$ induced by $V^*$ will likely be substantially more dense than all other
$k$-subgraphs of $G$ if $p + q < 1$.
We wish to determine which choices of $p, q$ and $k$ yield $G$ 
such that the planted dense $k$-subgraph $G(V^*)$ can be recovered from the optimal solution of \eqref{eq: dks relaxation}.
Note that $V^*$ is a $k$-clique of $G$ if $q=0$. 
Theorem~7 of \cite{ames2011nuclear} states that a planted $k$-clique can be recovered from the optimal solution of \eqref{eq: AV prob}
with high probability if $|V| = O(k^2)$ in this case.
The following theorem generalizes this result for all $q \neq 0$.

\begin{theorem} \label{thm: dks recovery}
	Suppose that the $N$-node graph $G$ is sampled from the planted dense $k$-subgraph model 
	 with $p,q$ and $k$ satisfying $p+q < 1$ and
	\begin{equation}
		 \label{eqn: pq assumptions}
		(1-p) k \ge \max \{8 p, \, 1\} \cdot 8\log k , \;\;\;\; \sqrt{p N} \ge (1-p)\log N
	\end{equation}
	\begin{equation}	
		(1 - p - q) k \ge 72 \max \bra{ \rbra{ q(1-q) k \log k }^{1/2}, \log k } \label{eqn: gap ass1}
	\end{equation}	
	Then there exist absolute constants  $c_1, c_2, c_3 > 0$ such that if
	\begin{align}
		(1-p-q) (1-p)  k & >  c_1 \max \bra{ p^{1/2}, \bbra{(1-p) k }^{-1/2} } \cdot N^{1/2} \log N \label{eqn: gap ass2}
	\end{align}	
	then the $k$-subgraph induced by $V^*$ is the densest $k$-subgraph of $G$ 
	and the proposed solution $(X^*, Y^*)$ is the unique optimal solution of \eqref{eq: dks relaxation} with high probability for 
	\begin{equation} \label{eqn: gamma choice}
		\gamma \in \rbra{ \frac{c_2}{(1-p-q) k },\, \frac{c_3}{(1-p-q) k } }.
	\end{equation}		
\end{theorem}

Here, and in the rest of the paper, an event is said to occur \emph{with high probability (w.h.p.)} if it occurs with probability tending
polynomially to $1$ as $k$ (or $\min\{k_1, k_2\}$ in Section~\ref{sec: dkb prob}) approaches $+\infty$.

The constant $c_1$ places a lower bound on the size of planted clique recoverable, while $c_2$ and $c_3$ provided a range of acceptable regularization parameters.
The assumption \eqref{eqn: gap ass2} places a lower bound on the density of the subgraph $G(V^*)$ induced by the planted clique in terms of its size $k$. On the other hand, the assumption \eqref{eqn: pq assumptions} is used to place bounds on the number
of noise edges in $G$; in particular, \eqref{eqn: pq assumptions} ensures that $p$ is not too close to $0$ or $1$, primarily
for the sake of simplifying the proof of Theorem~\ref{thm: dks recovery}.
If \eqref{eqn: pq assumptions} and \eqref{eqn: gap ass2} are satisfied, then
Theorem~\ref{thm: dks recovery} states that if we set the regularization parameter $\gamma$ equal to
$
	\gamma = {\kappa}/ ((1- p - q)k)
$
for any $\kappa$ in the interval $(c_2, c_3)$, then we can recover a planted $k$-clique of $G$ w.h.p.~from the optimal solution of \eqref{eq: dks relaxation} provided that the size of the planted clique $k$ is greater than a constant $c_1$ times the maximum
of $\sqrt{p N} \log N$ and $\sqrt{N/((1-p) k)} \log N$. 
That $\gamma$ is chosen from an interval and not a single value as in Theorem~\ref{thm: dks adv guarantee} is a
consequence of the probabilistic analysis in Section~\ref{sec: proof}; see Sections~\ref{sec: F bound}
and \ref{sec: W bound}.
 The fact that we deterministically construct the graph $G$
in Theorem~\ref{thm: dks adv guarantee} allows us to choose a single value of $\gamma$ in order to simplify our analysis substantially; the proof of Theorem~\ref{thm: dks adv guarantee} can be modified to ensure recovery
for a range of $\gamma$ as in the probabilistic case (see the analysis in the supplemental material found 
in \cite{DKSapps}).

To further clarify the implications of Theorem~\ref{thm: dks recovery} we consider two cases.
Suppose that the graph $G$ constructed according to ($\rho_1$) and ($\rho_2$) is dense; that is both $p$ and $q$ are fixed with respect to $k$ and $N$.
 In this case,  Theorem~\ref{thm: dks recovery}  suggests that $G(V^*)$ is the densest $k$-subgraph and its matrix representation is the unique optimal solution of
 \eqref{eq: dks relaxation} w.h.p.~provided that $k$ is at least as large as $\Omega(\sqrt{N} \log N)$.
 This lower bound matches that of \cite[Theorem~7]{ames2011nuclear}, as well as 
\cite{kuvcera1995expected,alon1998finding,feige2000finding,mcsherry2001spectral,feige2010finding,dekel2010finding,deshpande2013finding}, up to the constant and logarithmic factors, despite
 the presence of additional noise in the form of edge deletions.
 Moreover, modifying the proof of Theorem~\ref{thm: dks recovery} to follow the  proof of  \cite[Theorem~7]{ames2011nuclear} shows that  the planted dense $k$-subgraph can be recovered w.h.p.\;
 provided $k=\Omega(\sqrt{N})$ in the dense case; see the remarks following Lemma~\ref{lem: R bound}.
 Whether planted cliques of size $o(\sqrt{N})$ can be recovered in polynomial-time is still an open problem, although this task is widely believed to be intractable
 (and this presumed hardness has been exploited in cryptographic applications \cite{juels2000hiding} and complexity analysis \cite{alon2007testing,hazan2011hard,alon2011inapproximability,berthet2013computational}). 
Moreover, a number of algorithmic approaches \cite{jerrum1992large,feige2003probable,nadakuditi2012hard} have been shown to fail to recover planted cliques of size $o(\sqrt{N})$ in polynomial-time.
 
When the noise obscuring the planted clique is sparse, i.e., both $p$ and $q$ are tending to $0$ as $N \ra\infty$, the lower bound on the size of a recoverable clique can be significantly improved.
For example, if $p, q = O(1/k)$ then Theorem~\ref{thm: dks recovery} states that the planted clique can be recovered w.h.p. if $k = \Omega(N^{1/3} \log N)$.
On the other hand, if either $p$ or $q$ tends to $1$ as $N \ra\infty$,  then the minimum size of $k$ required
for exact recovery will necessarily increase.

It is important to note that  the choice of $\gamma$ in both Theorem~\ref{thm: dks adv guarantee} and Theorem~\ref{thm: dks recovery}  ensuring exact recovery is not universal, but rather depends on 
the parameters governing  edge addition and deletion.
These quantities are typically not known in practice.
However, under stronger assumptions on the edge corrupting noise, $\gamma$  independent of the unknown noise parameters may be identified.
For example,  if we impose the stronger assumption that  $p + q \le 1/2$, then we may take $\gamma = 6/k$.

\section{The Densest $(k_1, k_2)$-Subgraph Problem}
\label{sec: dkb prob}
Let ${G = (U,V,E)}$ be a bipartite graph. That is, ${G}$ is a graph whose vertex set can be partitioned into
two independent sets $U$ and $V$.
We say that a bipartite subgraph $H = (U',V', E')$  is
a \textit{$(k_1 k_2)$-subgraph} of $G$ if $U' \subseteq U$ and $V'\subseteq V$ such that  $|U'|\cdot|V'| = k_1 k_2$.
Given bipartite graph $G$ and integers $k_1, k_2$, the \emph{densest $(k_1, k_2)$-subgraph problem} seeks  the $(k_1, k_2)$-subgraph of $G$ containing
maximum number of edges.
This problem is NP-hard, by reduction from the maximum edge biclique problem \cite{peeters2003maximum}, and hard to approximate \cite{feige2002relations,goerdt2004approximation}.

As before, we consider a convex relaxation of the densest $(k_1, k_2)$-subgraph problem motivated by the fact that the adjacency matrices of dense $(k_1, k_2)$-subgraphs
are closely approximated by rank-one matrices.
If $(U', V')$ is a biclique of $G$, i.e. $ij \in E$ for all $i\in U'$, $j \in V'$ then the bipartite subgraph $G(U',V')$ induced by $(U',V')$
is a $(k_1,k_2)$-subgraph of $G$, containing all $k_1 k_2$ possible edges between $U'$ and $V'$.
In this case, the $(U,V)$ block of the adjacency matrix of $G(U',V')$ is equal to $X' = \bubv$, where $\u$ and $\bv$ are the characteristic vectors of $U'$ and $V'$ respectively.
Note that $X'$ has rank equal to one.
If $(U',V')$ is not a biclique of $G$, then there exists some $i \in U'$, $j\in V'$ such that $ij \notin E$. In this case,
the $(U,V)$ block of the adjacency matrix of $G(U',V')$ has the form $X' + Y'$, where $Y' = -P_{\tilde E}(X\rq{})$.
Here $P_{\tilde E}$ denotes  the orthogonal projection onto the set of matrices with support contained in the complement  $\tilde E := (U\times V) - E$ of the edge set $E$.
As such, the densest $(k_1,k_2)$-subgraph problem may be formulated as the rank constrained cardinality minimization problem:
\begin{align*}
	\min \left\{ \| Y \|_0: \rank X = 1, \; 
					\e^T X \e = k_1 k_2, \;
					X_{ij} + Y_{ij} = 0, \; \forall \, ij \in \tilde E, \;
					X \in \{0,1\}^{U\times V} \right\}.
\end{align*}
This problem is identical to \eqref{eq: dks min l0} but for a slightly different definition of the set $\tilde E$, a different right-hand side in the sum constraint,
and omission of symmetry constraints.
As before, we obtain a tractable convex relaxation by  moving the rank constraint to the objective as a regularization term,  relaxing rank and the $\ell_0$ norm with
the nuclear norm and $\ell_1$ norm, respectively, and replacing the binary constraints with appropriate box constraints:
\begin{equation}	\label{eq: dkb relaxation}
	\min \left\{ \|X\|_*+ \gamma \|Y\|_1 :					
					\e^T X \e = k_1 k_2, \;
					X_{ij} + Y_{ij} = 0, \; \forall \, ij \in \tilde E, \;
					X \in [0,1]^{U\times V} \right\}.
\end{equation}
Again, except for superficial differences, this problem is identical to the convex relaxation of \eqref{eq: dks min l0} given by \eqref{eq: dks relaxation}.
As can be expected, the recovery guarantees for the relaxation of the densest $k$-subgraph problem translate to similar guarantees for the
convex relaxation \eqref{eq: dkb relaxation} of the densest \kbs\,  problem.

As in Section~\ref{sec: dks prob}, we consider the performance of the relaxation \eqref{eq: dkb relaxation} in the special case that the input graph contains an
especially dense $(k_1,k_2)$-subgraph.
As before, we consider graphs constructed to contain such a subgraph as follows. First, all edges between $U^*$ and $V^*$ are added for
a particular pair of subsets $U^* \subseteq U$, $V^* \subseteq V$ such that $|U^*| = k_1$ and $|V^* |= k_2$.
Then some of the remaining potential edges in $U\times V$ are added while some of the edges between $U^*$ and $V^*$ are deleted.
As in the previous section, this introduction of noise is either performed deterministically by an adversary or at random with each edge added or deleted independently with fixed probability.
The following theorem provides bounds on the amount of deterministic noise ensuring exact recovery of the planted dense \kbs\, by \eqref{eq: dkb relaxation}.

\begin{theorem} \label{thm: kdb det case}
	Let $G = (U,V,E)$ be a bipartite graph and let $U^* \subseteq U$, $V^*\subseteq V$ be subsets of cardinality $k_1$ and $k_2$ respectively.
	Let $\u$ and $\bv$ denote the characteristic vectors of $U^*$ and $V^*$, and
	let $ (X^*,Y^*) = (\bubv, -P_{\tilde E}(\bubv))$.
	Suppose that $G(U^*, V^*)$ contains at least $k_1 k_2 - s$ edges and that $G$ contains at most $r$ edges other than those in $G(U^*, V^*)$.
	Suppose that every node in $V^*$ is adjacent to at least $(1 - \alpha_1) k_1$ nodes in $U^*$ and every node in $U^*$ is adjacent to at least $(1 - \alpha_2) k_2$
	nodes in $V^*$ for some scalars $\alpha_1, \alpha_2 > 0$.
	Further, suppose that each node in $V-V^*$ is adjacent to at most $\beta_1 k_1$ nodes in $U^*$ and each node in $U-U^*$ is adjacent to at most $\beta_2 k_2$ nodes in $V^*$ for
	some $\beta_1, \beta_2 > 0$.
	Finally suppose that the scalars $\alpha_1, \alpha_2, \beta_1, \beta_2$ satisfy
	$
		\alpha_1 + \alpha_2 + \max\{ \beta_1, \beta_2\} < 1.
	$
	Then there exist scalars $c_1, c_2 > 0$, depending only on $\alpha_1, \alpha_2, \beta_1,$ and $\beta_2$, such that if
	$	r \le c_1 k_1 k_2$ and $s \le c_2 k_1 k_2$
	then
	$G(U^*, V^*)$ is the unique maximum density $(k_1, k_2)$-subgraph of $G$ and $(X^*,Y^*) $ is the unique optimal solution of \eqref{eq: dkb relaxation} for	
	$
		\gamma = 2 \bbra{{\sqrt{k_1 k_2} } \rbra{ 1 - \alpha_1 - \alpha_2 - \max\{\beta_1, \beta_2\} }  }^{-1}.
	$
\end{theorem}

Here, $(\alpha_1, \alpha_2)$ are analogous to $\delta_1$ and $(\beta_1, \beta_2)$ are analogous to $\delta_2$ in Theorem~\ref{thm: dks adv guarantee}. That is, Theorem~\ref{thm: kdb det case} implies that we may recover the densest $(k_1, k_2)$-subgraph of $G$ provided this densest $(k_1, k_2)$-subgraph is sufficiently dense. As in the earlier theorem,
``sufficiently dense" is controlled by the minimum degree of nodes in $G(U^*, V^*)$, parametrized by $(\alpha_1, \alpha_2)$ and $c_2$, and the maximum number of neighbours
outside of $G(U^*,V^*)$ that each node in $U^*$ and $V^*$ may have, as parametrized by $(\beta_1, \beta_2)$ and $c_1$.

As before, the bounds on the number of edge corruptions that guarantee exact recovery given by Theorem~\ref{thm: kdb det case} are identical to those provided in \cite[Section~5.1]{ames2011nuclear} (up to constants).
Moreover, these bounds are optimal for reasons similar to those in the discussion immediately following Theorem~\ref{thm: dks adv guarantee}.
	

A similar result holds for random bipartite graphs $G = (U,V,E)$ constructed as follows:
\begin{itemize}
	\item[$(\psi_1)$] For some $k_1$-subset $U^* \subseteq U$ and $k_2$-subset $V^*\subseteq V$, we add each potential edge from $U^*$ to $V^*$ independently with  probability $1-q$.
	\item[($\psi_2$)] Then each remaining possible edge is added independently to $E$ with  probability $p$.
\end{itemize}

By construction $G(U^*, V^*)$ is dense  in expectation, relative to its complement, if  $p + q < 1$.
Theorem~9 of \cite{ames2011nuclear} asserts that $G(U^*,V^*)$ is the densest $(k_1,k_2)$-subgraph of $G$ and can be recovered
using a modification of \eqref{eq: dkb relaxation}, for sufficiently large $k_1$ and $k_2$, in the special case that $q = 0$.
The following theorem generalizes this result for all $p$ and $q$.

\begin{theorem}	\label{thm: kdb random case}
	Suppose that the $(N_1, N_2)$-node bipartite graph $G = (U,V,E)$ is constructed according to $(\Psi_1)$ and $(\Psi_2)$ such that $p + q < 1$ and
	\begin{align}
		(1 - p) k_i & \ge \max\{8, 64 p \} \log k_i \\ 
		p N_i & \ge (1-p)^2 \log^2 N_i 		\\
		(1 - p -q) k_i &\ge 72 \max\bra{\log k_i,  \bbra{ q(1-q) k_i \log k_i}^{1/2} } 
	\end{align}
	for $i=1,2$.
	Then there exist absolute constants $c_1, c_2, c_3 > 0$ such that if
	\begin{equation}	\label{ass: gap W}
		c_1 (1-p-q) \sqrt{k_1k_2} \ge \bar N^{1/2} \log \bar N \cdot\max \bra{ p^{1/2}, ( (1-p) \min\{k_1,k_2\} )^{-1/2} },
	\end{equation}			
	where $\bar N = \max\{N_1, N_2\}$,
	then $G(U^*,V^*)$ is the densest $(k_1, k_2)$-subgraph  of $G$
	and $(X^*,Y^*)$ is the unique optimal solution of \eqref{eq: dkb relaxation} w.h.p.~for 
	\begin{equation}	\label{eq: dkb gamma range}
		\gamma \in \rbra{ \frac{c_2}{(1 - p-q)\sqrt{k_1 k_2}},  \; \frac{c_3}{(1 - p-q)\sqrt{k_1 k_2} }}.
	\end{equation}			
\end{theorem}

Theorem~\ref{thm: kdb random case} is the bipartite analogue of Theorem~\ref{thm: dks recovery}. That is,
 Theorem~\ref{thm: kdb random case} implies that we can recover
the planted $(k_1, k_2)$-subgraph of $G$ from the optimal solution of \eqref{eq: dkb relaxation} provided this subgraph is sufficiently large, as characterized by \eqref{ass: gap W},
and  we choose the regularization parameter $\gamma$ from a particular interval of acceptable values given by \eqref{eq: dkb gamma range}.

\section{Exact Recovery of the Densest $k$-Subgraph Under Random Noise}
\label{sec: proof}
This section consists of a proof of Theorem~\ref{thm: dks recovery}.
The proofs of Theorems~\ref{thm: dks adv guarantee},~\ref{thm: kdb det case},~and~\ref{thm: kdb random case}
follow a similar structure and are omitted; proofs of these theorems may be found in the supplemental material
\cite{DKSapps}.
Let $G = (V,E)$ be a graph sampled from the planted dense $k$-subgraph model.
Let $V^*$ be the node set of the
planted dense $k$-subgraph $G(V^*)$ and $\bv$ be its characteristic vector.
Our goal is to show that the solution $(X^*, Y^*) := (\bv\bv^T, -P_{\tilde E}(\bv\bv^T))$,
where $\tilde E := (V\times V) - (E\cup \{uu: u\in V\})$, is the unique optimal solution of \eqref{eq: dks relaxation}
and $G(V^*)$ is the densest $k$-subgraph of $G$ in the case that $G$ satisfies the hypothesis of Theorem~\ref{thm: dks recovery}.
To do so, we will apply the Karush-Kuhn-Tucker theorem to derive sufficient conditions for optimality of a feasible solution
of \eqref{eq: dks relaxation} corresponding to a $k$-subgraph of $G$ and then establish that these sufficient conditions
are satisfied at $(X^*, Y^*)$ with high probability if the assumptions of Theorem~\ref{thm: dks recovery} are satisfied.

\subsection{Optimality Conditions}
We will show that $(X^*, Y^*)$ is optimal for \eqref{eq: dks relaxation} and, consequently, $G(V^*)$ is the densest $k$-subgraph of $G$ by 
establishing that $(X^*,Y^*)$ satisfy the sufficient conditions for optimality given by the Karush-Kuhn-Tucker theorem (see \cite[Section 5.5.3]{boydvdb}).
The following theorem provides the necessary specialization of these conditions to \eqref{eq: dks relaxation}. A proof of Theorem~\ref{thm: KKT conds} can be found in Appendix~\ref{app: KKT proof}.

\begin{theorem}	\label{thm: KKT conds}
	Let $G = (V,E)$ be a graph sampled from the planted dense $k$-subgraph model.
	Let $\bar V$ be a subset of $V$ of cardinality $k$ and let $\bar \bv$ be the characteristic vector of $\bar V$.
	Let $\bar X = \barX$ and let $\bar Y$ be defined as in \eqref{eq: proposed Y}.
	Suppose that there exist $F, W \in \R^{V\times V}$, $\lambda \in \R$, and $M \in  \R^{V\times V}_{+}$ such that
	\begin{align}
		\label{eq: KKT dual feas}
			{\bar X}/{k} + W -\lambda \e\e^T- \gamma( \bar Y + F ) + M = 0, \\
		\label{eq: KKT W def}
			W \bar\bv = W^T \bar \bv = 0, \; \|W\| \le 1 , \\
		\label{eq: KKT F def}	
			P_{\Omega} (F) = 0, \;\;\ \|F\|_\infty \le 1, \\
		\label{eq: KKT CS edges}
			F_{ij} = 0 \;\; \mbox{for all } (i,j) \in E \cup \{vv: v \in V\} ,\\
		\label{eq: KKT CS upperbound}
			M_{ij} = 0 \;\; \mbox{for all } (i,j) \in (V\times V) - (\bar V \times \bar V).
	\end{align}
	Then $(\bar X, \bar Y)$ is an optimal solution of \eqref{eq: dks relaxation} and
	 the subgraph $G(\bar V)$ induced by $\bar V$ is a maximum density $k$-subgraph of $G$.
	Moreover, if $\|W\| < 1$ and $\|F\|_\infty < 1$ then
	$(\bar X, \bar Y)$ is the unique optimal solution of \eqref{eq: dks relaxation} and
	$G(\bar V)$ is the unique maximum density $k$-subgraph of $G$.
\end{theorem}

It remains to show that multipliers $F, W \in \R^{V\times V}$, $\lambda \in \R$, and $M\in  \R^{V\times V}_{+}$
corresponding to the proposed solution $(X^*, Y^*)$ and satisfying the hypothesis of Theorem~\ref{thm: KKT conds}  do indeed exist. 
In particular,			
we consider $W$ and  $F$ constructed according to the following cases:
\begin{itemize}
	\item[($\omega_1$)]
		If $(i,j) \in V^* \times V^*$ such that $ij \in E$ or $i=j$, choosing
		$
			W_{ij} = \tilde \lambda - M_{ij},
		$
		where $\tilde \lambda := \lambda - 1/k$,
		ensures that the left-hand side of \eqref{eq: KKT dual feas} is equal to $0$.
	\item[($\omega_2$)]
		If $(i,j) \in \Omega = V^*\times V^* \cap \tilde E$, then $F_{ij} = 0$ and choosing
		$W_{ij} = \tilde \lambda - \gamma - M_{ij}$ makes the left-hand side of \eqref{eq: KKT dual feas} equal to $0$.
	\item[($\omega_3$)]
		Let $(i,j) \in (V\times V) - (V^*\times V^*)$ such that $ij \in E$ or $i = j$, then the left-hand side of \eqref{eq: KKT dual feas}
		is equal to $W_{ij} - \lambda$. In this case, we choose $W_{ij} = \lambda$ to make both sides of \eqref{eq: KKT dual feas} zero.
	\item[($\omega_4$)]
		Suppose that $i,j \in V-V^*$ such that $(i,j) \in \tilde E$. We choose 
		$$
			W_{ij} = - \lambda \rbra{\frac{p}{1-p}} , \;\;\; F_{ij} = - \frac{\lambda}{\gamma} \rbra{ \frac{1}{1-p} } .
		$$
		Again, by our choice of $W_{ij}$ and $F_{ij}$, the left-hand side of \eqref{eq: KKT dual feas} is zero.
	\item[($\omega_5$)]
		If $i \in V^*$, $j \in V-V^*$ such that $(i,j) \in \tilde E$ then we choose
		$$
			W_{ij} = - \lambda \rbra{ \frac{n_j}{k- n_j} }, \;\;\; F_{ij} = -\frac{\lambda}{\gamma} \rbra{ \frac{k}{k - n_j} } 
		$$
		where $n_j$ is equal to the number of neighbours of $j$ in $V^*$.
	\item[($\omega_6$)]
		If $i \in V-V^*$, $j\in V^*$ such that $(i,j) \in \tilde E$ we choose $W_{ij}, F_{ij}$ symmetrically according to ($\omega_5$); 
		that is, we choose $W_{ij} = W_{ji}$, $F_{ij} = F_{ji}$ in this case.
\end{itemize}
It remains to construct multipliers $M$, $\lambda$, and $\gamma$ so that 
if $p,q,$ and $k$ satisfy the hypothesis of Theorem~\ref{thm: dks recovery}
then$W$ and $F$ as chosen above satisfy \eqref{eq: KKT W def} and \eqref{eq: KKT F def}.
In this case, $(X^*, Y^*)$ is optimal for \eqref{eq: dks relaxation} and the corresponding densest $k$-subgraph of $G$ can be recovered from $X^*$.
The remainder of the proof is structured as follows.
In Section~\ref{sec: mu}, we construct valid $\lambda \in \R$ and $M \in \R^{V\times V}_+$ such that $W\bv = W^T \bv = 0$.
We next establish  that $\|F\|_\infty  < 1$ w.h.p. for this choice of $\lambda$ and $M$ and a particular choice of regularization parameter $\gamma$ in Section~\ref{sec: F bound}.
We conclude in Section~\ref{sec: W bound} by showing that $\|W\| <  1$  w.h.p.~provided the assumptions of Theorem~\ref{thm: dks recovery} are satisfied.

\subsection{Choice of the Multipliers $\lambda$ and $M$}
\label{sec: mu}

In this section, we construct multipliers $\lambda \in \R$ and $M\in \R^{V\times V}_+$
such that $W\bv = W^T \bv = 0$.
Note that $[W\bv]_i = \sum_{j \in V^*} W_{ij}$ for all $i \in V$.
If $i \in V-V^*$, we have
$$
	[W\bv]_i = n_i \lambda - (k - n_i) \rbra{ \frac{n_i}{k-n_i} } \lambda = 0
$$
by our choice of $W_{ij}$ in ($\omega_3$) and ($\omega_5$). By symmetry, $[W^T\bv]_i = 0$ for all $i \in V-V^*$.

The conditions $W(V^*, V^*) \e = W(V^*,V^*)^T \e= 0$ define $2k$ equations for the $k^2$ unknown entries of  $M$. 
To obtain a particular solution of this underdetermined system, we parametrize $M$ as $M = \y\e^T + \e \y^T$,
for some $\y \in \R^V$.
After this parametrization
$$
	\sum_{j\in V^*} W_{ij} = k \tilde\lambda  - (k-1 -n_i)\gamma- k y_i - \e^T \y.
$$
Rearranging shows that $\y$ is the solution of the linear system
\begin{equation}	\label{eq: y system}
	(kI + \e\e^T)\y =  k\tilde\lambda \e - \gamma( (k-1) \e - \n) ,
\end{equation}
where $\n \in \R^{V^*}$ is the vector with $i$th entry $n_i$ equal to the degree of node $i$ in $G(V^*)$.
By the Sherman-Morrison-Woodbury formula \cite[Equation (2.1.4)]{GV}, we have
\begin{align}	
	\y 
		= \frac{1}{2k} \rbra{ k  \tilde \lambda- (k-1) \gamma }\e + \frac{\gamma}{k} \rbra{ \n - \rbra{ \frac{\n^T \e}{2k} } \e } .
		\label{eq: y formula}
\end{align}
and
$
	\E [\y] 
			= \bbra{ k  \tilde \lambda - (k-1) \gamma q }\e /(2k)	\label{eq: Ey formula}
$ 
by the fact that $\E[\n] = (k-1)(1-q) \e$.
Taking $\lambda = \gamma (\epsilon + q) + 1/k$ yields
$
	\E[\y] = 
	\bbra{k \epsilon + q } \e /(2k).
$
Therefore, each entry of $\y$ and, consequently, each entry of $M$ is positive in expectation for all $\epsilon > 0$.
Therefore, 
 it suffices to show that
\begin{equation} \label{eqn: yinf error bound}
	\| \y - \E[\y] \|_\infty = \frac{\gamma}{k}\left\| \n - \rbra{ \frac{\n^T \e}{2k} } \e  -  \E \sbra{ \n - \rbra{ \frac{\n^T \e}{2k} } \e }  \right\|_\infty
	\le  \frac{\gamma}{2 k} \bbra{k \epsilon + q }
\end{equation}
with high probability to establish that the entries of $M$ are nonnegative with high probability, by the fact that each component $y_i$ is bounded below by $\E[y_i] - \| \y - \E[\y] \|_\infty$.
To do so, we will use the following concentration bound on  the sum of independent Bernoulli variables.

\begin{lemma} \label{lem: Bernstein bound}
	Let $x_1, \dots, x_m$ be a sequence of $m$ independent Bernoulli trials, each succeeding with probability $p$ and
	let $s = \sum_{i=1}^m x_i$ be the binomially distributed variable describing the total number of successes.
	Then
	$
		|s - p m | \le 6 \max \bra{ \sqrt{p(1-p) m \log m}, \log m  }
	$
	with probability at least $1 - 2 m^{-12}$.
\end{lemma}	 	

Lemma~\ref{lem: Bernstein bound} is a specialization of the standard Bernstein inequality (see, for example,\cite[Theorem~6]{lugosi2009})
to binomially distributed random variables; the proof is left to the reader.
	
We are now ready to state and prove the desired lower bound on the entries of $\y$.

\begin{lemma}
	\label{lem: bound on y}	
	For each $i \in V^*$, we have
	\begin{equation} \label{eq: y bound}
		 y_i \ge  \gamma \rbra{ \frac{\epsilon }{2}  - 12 \max \bra{ \rbra{ \frac{q(1-q) \log k}{k} }^{1/2}, \frac{\log k}{k} }}.
	\end{equation}
	with high probability.
\end{lemma}

\begin{proof}
Each entry  of $\n$ corresponds to $k-1$ independent Bernoulli trials, each with probability of success $1-q$.
Applying Lemma~\ref{lem: Bernstein bound} and the union bound shows that
\begin{equation} \label{eq: ni bounds}
	| n_i - (1-q)(k-1) | \le 6 \max\bra{ \bbra{ q(1-q) k \log k }^{1/2}, \log k }
\end{equation}
for all $i \in V^*$
with high probability.
On the other hand, $\n^T \e = 2|E(G(V^*))|$ because each entry of $\n$ is equal to the degree in the subgraph induced by $V^*$ of the corresponding node.
Therefore $\n^T \e$ is a binomially distributed random variable corresponding to ${k \choose 2}$ independent Bernoulli trials, each with probability of success $1-q$.
As before, Lemma~\ref{lem: Bernstein bound} implies that
\begin{equation}	\label{eq: ne bound}
	| \n^T \e - \E[\n^T \e] | \le 
		12 \max\bra{ k \bbra{ q(1-q)  \log k}^{1/2}, \;2 \log k }
\end{equation}
with high probability.
Substituting \eqref{eq: ni bounds} and \eqref{eq: ne bound} into the left-hand side of \eqref{eqn: yinf error bound} and applying the triangle inequality shows that
\begin{equation}	\label{eq: y inf bound}
	\|\y - \E[\y] \|_\infty 
	\le \frac{12 \gamma}{k}  \max\bra{ \rbra{  q(1-q) k\log k}^{1/2}, \; \log k }
\end{equation}
for sufficiently large $k$ with high probability.
Subtracting the right-hand side of \eqref{eq: y inf bound} from $E[y_i] \ge \gamma \epsilon/2$ for each $i \in V^*$ completes the proof.
\qed
\end{proof}

In Section~\ref{sec: F bound}, we will choose $\epsilon = (1-p-q)/3$ to ensure that $\|F\|_\infty \le 1$ with high probability.
Substituting this choice of $\epsilon$ in the right-hand side of \eqref{eq: y bound} yields
$$
	\min_{i\in V^*} y_i \ge\frac{\gamma}{6} \rbra{ (1-p-q) - \frac{72}{k} \max\bra{ k \bbra{ q(1-q)  \log k}^{1/2}, \;2 \log k } }
$$	
with high probability.
Therefore, the entries of the multiplier $M$ are  nonnegative w.h.p.~if $p,q,$ and $k$ satisfy \eqref{eqn: gap ass1}.

\subsection{A Bound on $\|F\|_\infty$}
\label{sec: F bound}
We next establish that $\|F\|_\infty < 1$ w.h.p.~under the assumptions of Theorem~\ref{thm: dks recovery}.
Recall that all diagonal entries, entries corresponding to edges in $G$, and entries indexed by $V^*\times V^*$ of $F$ are chosen to be equal to $0$.
It remains to bound $|F_{ij}|$ when $ij \notin E$, and $(i,j) \in (V\times V) - (V^*\times V^*)$.

We first consider the case when $i,j \in V-V^*$ and $ij \notin E$. In this case, we choose $F_{ij}$ according to ($\omega_4$):
$F_{ij} = -\lambda/(\gamma(1-p))$.
Substituting $\lambda = \gamma( \epsilon + q ) + 1/k$, we have $|F_{ij} | \le 1$ if and only if
\begin{equation}	\label{eq: F bound, case 4}
	\frac{1}{\gamma k} + \epsilon + p + q \le 1.
\end{equation}
Taking $\epsilon = (1 - p - q)/3$ and $\gamma \ge 1/(\epsilon k)$ ensures that \eqref{eq: F bound, case 4} 
is satisfied in this case.

We next consider  $i \in V^*$, $j\in V-V^*$ such that $ij \notin E$. The final case, $i\in V-V^*$, $j\in V^*$, $ij \notin E$, follows immediately
by symmetry. In this case, we take
$ F_{ij} = - \lambda k/(\gamma(k-n_j))$
by ($\omega_5$).
Clearly, $|F_{ij}| \le 1$ if and only if
\begin{equation}	\label{eq: F bound, case 5}
	\frac{1}{\gamma k} + \epsilon + q + \frac{n_j}{k} \le 1.
\end{equation}
Applying Lemma~\ref{lem: Bernstein bound}  and the union bound over all $j \in V-V^*$
shows that
$$
	|n_j - p k| \le  6 \max \bra{ \rbra{p(1-p) k \log k}^{1/2}, \log k }
$$
for all $j \in V-V^*$
with high probability.
Thus, the left-hand side of \eqref{eq: F bound, case 5}
is bounded above by
$$
	\frac{1}{\gamma k} + \epsilon + q + p + \frac{6}{k} \max \bra{ \rbra{p(1-p) k \log k}^{1/2}, \log k }
$$
with high probability, which is bounded above by $1$ for sufficiently large $k$ for our choice
of $\epsilon$ and $\gamma$.
Therefore, our choice of $F$ satisfies $\|F\|_\infty < 1$ with high probability. 

a

\subsection{A Bound on $\|W\|$}
\label{sec: W bound}
We complete the proof by establishing that $\|W\|$ is bounded above w.h.p.~by a multiple of $\sqrt{N \log N}/k$ for $\gamma$, $\lambda$, and $M$  chosen as in Sections~\ref{sec: mu}~and~\ref{sec: F bound}.
Specifically, we have the following bound on $\|W\|$.

\begin{lemma}\label{thm: W bound}
	Suppose that $p, q$, and $k$ satisfy \eqref{eqn: pq assumptions}.
	Then 
	\begin{align*}
		\|W \| \le& 24 \gamma \max\bra{\rbra{ q(1-q) k \log k}^{1/2}, \log^2 k} \\
					&+ 36 \lambda  \max \bra{ 1, \rbra{p(1-p) k}^{1/2}} \rbra{ \frac{N}{(1-p)^3 k^3 }}^{1/2} \log N 
	\end{align*}					
	with high probability.
\end{lemma}

Taking $\gamma = O\bbra{ \bbra{ (1-p-q) k }^{-1}}$ and $\lambda = 1/k + \gamma ((1-p-q)/3 + q)$ shows that
$$
	\|W\|
			= O \rbra{ \max \bra{ 1, \rbra{p(1-p) k}^{1/2}} \rbra{ \frac{N}{(1-p)^3 k^3}}^{1/2} \log N   }
$$			
with high probability.
Therefore $\|W\| < 1$ w.h.p.~if $p, q,$ and $k$ satisfy the assumptions of Theorem~\ref{thm: dks recovery} for appropriate choice of constants $c_1$ and $c_3$.

The remainder of this section comprises a proof of Lemma~\ref{thm: W bound}.
We decompose $W$ as $W = Q + R$, where
\begin{align*}
	Q_{ij} &= \branchdef{  W_{ij}, &\mbox{if } i,j \in V^* \\ 0, &\mbox{otherwise}  }  \hspace{0.5in}
	R_{ij} = \branchdef{ 0, & \mbox{if } i,j \in V^* \\ W_{ij}, &\mbox{otherwise.} }
\end{align*}
We will bound $\|Q\|$ and $\|R\|$ separately, and then apply the triangle inequality to obtain the
desired bound on $\|W\|$.
To do so, we will make repeated use of the following bound on the norm of a random symmetric matrix with
i.i.d.~mean zero entries.

\begin{lemma} \label{lem: Tropp}
	Let $A = [a_{ij}]\in \Sigma^n$ be a random symmetric matrix with i.i.d.~mean zero entries $a_{ij}$ with variance
	$\sigma^2$ and satisfying $|a_{ij}| \le B$. Then
	$
		\|A\| \le 6 \max \bra{ \sigma \sqrt{n \log n}, B \log^2 n }
	$
	with probability at least $1 - n^{-8}$.	
\end{lemma}

The proof of Lemma~\ref{lem: Tropp} follows from an application of the Noncommutative Bernstein Inequality \cite[Theorem~1.4]{tropp2011user} and is included as Appendix~\ref{app: bernstein proof}.

The following lemma gives the necessary bound on $\|Q\|$.

\begin{lemma}	\label{lem: Q bound}
	The matrix $Q$  satisfies
	 $
	 	\|Q\| \le 24 \gamma \max \{ \bbra{q(1-q)k\log k}^{1/2}, \log^2 k \}
	$
	with high probability.
\end{lemma}

\begin{proof}
	We have $\|Q\| = \|Q(V^*,V^*)\|$ by the block structure of $Q$.
	Let
	\begin{align*}
		Q_1 = H(V^*, V^*) - \rbra{\frac{k-1}{k}} q &\e\e^T, \;\;\;
		Q_2 = \frac{1}{k} \bbra{ \n \e^T - (1-q)(k-1) \e\e^T }, \;\;\;
		Q_3 = Q_2^T \\
		Q_4& = \frac{1}{k} \bbra{ \n^T \e - (1-q)(k-1) k },
	\end{align*}
	where $H$ is the adjacency matrix of the complement of  $G(V^*)$.
	Note that $Q(V^*, V^*) =  \sum_{i=1}^4 \gamma Q_i$. 
	We will bound  each $Q_i$ separately and then apply the triangle inequality to obtain the desired bound on $\|Q\|$.
	
	We begin with $\|Q_1\|$.  Let $\tilde H \in \Sigma^{V^*}$ be the random matrix  with
	off-diagonal entries equal to the corresponding entries of $H$ and whose diagonal entries are independent Bernoulli variables, each with probability of success equal to $q$.
	Then $\E[\tilde H ] = q \e\e^T$ and $\tilde H - q\e\e^T$ is a random symmetric matrix with i.i.d.~mean zero entries with variance equal to $\sigma^2 = q(1-q)$.
	Moreover, each entry of $\tilde H - q \e\e^T$ has magnitude	bounded above by $B = \max\{q, 1-q\} \le 1$.
	Therefore, applying Lemma~\ref{lem: Tropp} shows that
	$\|\tilde H - q \e\e^T \| \le 6 \max\{\sqrt{q(1-q) k \log k}, \log^2 k \}
	$
	with high probability.
	It follows immediately that
	\begin{align}	
		\|Q_1\| &\le \|\tilde H - q \e\e^T \| + \|(q/k) \e\e^T \| + \|\Diag (\diag \tilde H) \| \notag\\
			&\le 6 \max\bra{ \rbra{q(1-q) k \log k}^{1/2}, \log^2 k }+ q + 1 \label{eq: Q1 bound}
	\end{align}
	with high probability by the triangle inequality.
	
	We next bound $\|Q_2\|$ and $\|Q_3\|$. 
	By \eqref{eq: ni bounds}, we have
	$$
		\| \n - \E[\n]\|^2 \le k \| \n - \E[\n]\|_\infty^2  \le 36 \max\bra{q(1-q) k^2 \log k, k \log^2 k}
	$$
	\whp.
	It follows that
	\begin{equation}	\label{eq: Q2 bound}
		\|Q_2 \| = \|Q_3\| \le \frac{1}{k} \|\n - \E[\n]\| \|\e\| \le 6\max\bra{ \rbra{q(1-q) k \log k}^{1/2}, \log k }
	\end{equation}
	with high probability.	
	Finally,
	\begin{align}	
		\|Q_4\| &\le \frac{1}{k^2} \abs{ \n^T\e - \E[\n^T\e] } \|\e\e^T\| 
				\le 12 \max\bra{ \rbra{q(1-q) \log k}^{1/2},  2\log k/k }	\label{eq: Q4 bound}
	\end{align}
	with high probability, where the last inequality follows from \eqref{eq: ne bound}.
	Combining \eqref{eq: Q1 bound}, \eqref{eq: Q2 bound}, and \eqref{eq: Q4 bound} and applying the union bound we have
	$
		\|Q\|\le 24 \gamma \max\{ \rbra{q(1-q) k \log k}^{1/2}, \log^2 k \}
	$
	with high probability.
	\qed
\end{proof}
	
The following lemma provides the necessary bound on $\|R\|$.
\begin{lemma}	\label{lem: R bound}
	Suppose that $p$ and $k$ satisfy \eqref{eqn: pq assumptions}.
	Then
	$$
		\|R\| \le  36 \lambda  \max \bra{ 1, \rbra{p(1-k) k}^{1/2}} \rbra{ \frac{N}{(1-p)^3 k }}^{1/2} \log N
	$$
	with high probability.
\end{lemma}

\begin{proof}
	We decompose $R$ as in the proof of Theorem~7 in \cite{ames2011nuclear}.
	Specifically, we let $R = \lambda(R_1 + R_2 + R_3 + R_4 + R_5)$ as follows.
	
	We first define $R_1$ by considering the following cases. In Case $(\omega_3)$ we take $[R_1]_{ij} = W_{ij}$.
	In Cases $(\omega_4), (\omega_5),$ and $(\omega_6)$ we take $[R_1]_{ij} = -p/(1-p)$.
	Finally, for all $(i,j) \in V^*\times V^*$ we take $[R_1]_{ij}$ to be a random variable sampled independently from the distribution
	$$
		[R_1]_{ij} = \branchdef{ 1, &\mbox{with probability } p, \\ -p/(1-p), &\mbox{with probability } 1-p. }
	$$	
	By construction, the entries of $R_1$ are i.i.d.~random variables taking value $1$ with probability $p$ and value $-p/(1-p)$ otherwise.
	Applying Lemma~\ref{lem: Tropp} shows that
	\begin{equation} \label{eq: R_1}
		\|R_1\| \le 6 \max \bra{ B \log^2 N , \rbra{ \rbra{ \frac{p}{1-p} } N \log N }^{1/2} }
	\end{equation}
	with high probability, where $B := \max\{1, p/(1-p)\}$.
	
	We next define $R_2$ to be the correction matrix for the $(V^*, V^*)$ block of $R$. That is, $R_2(V^*, V^*) = -R_1(V^*, V^*)$ and $[R_2]_{ij} = 0$ if $(i,j) \in (V\times V) - (V^*\times V^*)$.
	Then
	$$
		\|R_2\| = \|R_1(V^*, V^*) \| \le 6 \max \bra{ B \log^2 k , \rbra{ \rbra{ \frac{p}{1-p} } k \log k }^{1/2} }
	$$
	with high probability by Lemma~\ref{lem: Tropp}.
	We define $R_3$ to be the correction matrix for diagonal entries of $R_1$:  $\lambda [R_3]_{ii} =  R_{ii} - \lambda [R_1]_{ii}$ for all $i \in V^*$.		
	By construction $R_3$ is a diagonal matrix with diagonal entries taking value either $0$ or $1/(1-p)$.  Therefore
	$
		\|R_3\| \le {1}/{1-p}.
	$
	
	Finally, we define $R_4$ and $R_5$ to be the correction matrices for Cases $(\omega_5)$ and $(\omega_6)$ respectively. 
	That is, we take $[R_4]_{ij} = p/(1-p) - n_j/(k-n_j)$ for all $i \in V^*$, $j\in V- V^*$ such that $ij \notin E$ and is equal to $0$ otherwise,
	and take	$R_5 = R_4^T$ by symmetry .
	Note that
	\begin{align*}
		\|R_4\|^2 &\le \|R_4\|^2_F = \sum_{j \in V-V^*} (k-n_j) \rbra{ \frac{ pk - n_j}{(1-p)(k - n_j)}}^2 = \sum_{j \in V-V^*} \frac{(n_j - pk)^2}{(1-p)^2(k-n_j)}.
	\end{align*}	
	By Lemma~\ref{lem: Bernstein bound},
	we have
	$
		|n_j - pk| \le 6  \max\{\sqrt{p(1-p) k \log k}, \log k \} 
	$
	with high probability. 
	Therefore,
	\begin{align*}
		\|R_4\|^2 &\le \frac{36 (N-k) \max\{{p(1-p) k \log k}, \log^2 k\} }{(1-p)^2 \rbra{ (1-p)k - 6\max\{\sqrt{p(1-p) k \log k}, \log k\}} }  \\
			&\le  \rbra{\frac{144 N}{(1-p)^3 k} } \max \{ p (1-p) k \log k, \log^2 k \}
	\end{align*}
	with high probability, where the last inequality follows from \eqref{eqn: pq assumptions},
	which implies that
	$$
		(1-p) k - 6 \max \{ p (1-p) k \log k, \log^2 k \} \ge \frac{1}{4} (1-p) k.
	$$
	Combining the upper bounds on each $\|R_i\|$ shows that
	$$
		\|R\| \le  36 \lambda  \max \bra{ 1, \rbra{p(1-k) k}^{1/2}} \rbra{ \frac{N}{(1-p)^3 k }}^{1/2} \log N 
	$$
	with high probability, provided $p,k,$ and $N$ satisfy \eqref{eqn: pq assumptions}. This completes the proof.	 \qed
\end{proof}	

The construction of the matrix $R$ is essentially identical to that of $W$ in \cite[Theorem~7]{ames2011nuclear}.
In the dense case, i.e., when $p$ is independent of $k$ and $N$, we may apply the proof of \cite[Theorem~7]{ames2011nuclear}
to show that $\|R\| = O( \sqrt{N} / k)$. This, in turn, suggests that we have exact recovery w.h.p. if $k = \Omega(\sqrt{N})$ in the dense case.

\section{Experimental Results}
\label{sec: expts}
In this section, we empirically evaluate the performance of our relaxation for the planted densest $k$-subgraph problem.
Specifically, we apply our relaxation \eqref{eq: dks relaxation} to $N$-node random graphs sampled from the planted dense $k$-subgraph model for a variety of planted clique sizes $k$.

For each randomly generated program input, we apply the Alternating Directions Method of Multipliers (ADMM) to solve \eqref{eq: dks relaxation}.
ADMM has recently gained popularity as an algorithmic framework for distributed convex optimization, in part, due to its being well-suited to large-scale problems arising
in machine learning and statistics. A full overview of ADMM and related methods is well beyond the scope of this paper; we direct the reader to the recent survey \cite{boyd2011distributed} and the references within for more details.
Note that we may also solve \eqref{eq: dks relaxation} by reformulating as the semidefinite program~\eqref{eq: dks sdp primal}
 and then solve this SDP using
interior point methods when the graph $G$ is small. However, the memory requirements needed to formulate and solve the Newton system corresponding to \eqref{eq: dks relaxation} are prohibitive for graphs containing more than a few hundred nodes.

A specialization of ADMM to our problem is given as Algorithm~\ref{alg: ADMM}; specifically, Algorithm~\ref{alg: ADMM} is a modification of the ADMM algorithm for Robust PCA given by \cite[Example~3]{goldfarb2010fast}.
We iteratively solve the linearly constrained optimization problem
$$
	\begin{array}{rl} \min 	& \|X\|_* + \gamma \|Y\|_1  + \id_{\Omega_Q}(Q) + \id_{\Omega_W}(W) + \id_{\Omega_Z }(Z) \\
				\st 		& X + Y = Q, \;\; X = W, \;\; X = Z,
	\end{array}
$$				
where $\Omega_Q := \{Q \in \R^{V\times V}:  P_{\tilde E}(Q) = 0\}$, $\Omega_{W} := \{W\in \R^{V\times V}: \e^TW\e = k^2\}$,
and $\Omega_Z := \{Z \in \R^{V\times V}: Z_{ij} \le 1 \; \forall (i,j)\in V\times V\}$.
Here $\id_S: \R^{V\times V} \ra \{0, +\infty\}$ is the indicator function of the set $S \subseteq \R^{V\times V}$, defined by
$\id_S(X) = 0$ if $X \in S$ and $+\infty$ otherwise.
During each iteration, we
sequentially update each primal decision variable by minimizing the augmented Lagrangian
\begin{align*}
	L_\tau  = & \|X\|_*  + \gamma \|Y\|_1 +  + \id_{\Omega_Q}(Q) + \id_{\Omega_W}(W) + \id_{\Omega_Z }(Z)   
	\\ & + \tr(\lambda_Q(X + Y - Q))  + \tr(\Lambda_W (X - W) )  + \tr(\Lambda_Z ( X - Z))   \\ &+ \frac{\tau}{2} \Bbra{ \|X+Y - Q\|^2 + \|X-W\|^2 + \|X - Z\|^2 }			
\end{align*}
in Gauss-Seidel fashion with respect to each primal variable  and then updating the dual variables $\lambda_Q, \lambda_W, \Lambda_Z$ using the updated primal variables.
Here $\tau$ is a regularization parameter chosen so that $L_\tau$ is strongly convex in each primal variable.
Equivalently, we update each of $X, Y, Q, W,$ and $Z$ by evaluation of an appropriate proximity operator during each iteration.
Minimizing the augmented Lagrangian with respect to each of the artificial primal variables $Q, W$ and $Z$ is equivalent to projecting onto each of the sets $\Omega_Q$, $\Omega_W$, and $\Omega_Z$, respectively;
each of these projections can be performed analytically.
On the other hand, the subproblems for updating $X$ and $Y$ in each iteration allow closed-form solutions via the elementwise soft thresholding operator $S_\phi: \R^n \ra \R^n$ defined by
$$
	\sbra{S_\phi (\x)}_{i} = \branchdef{ x_i - \phi, &\mbox{if } x_i > \phi \\ 0, & \mbox{if } -\phi \le x_i \le \phi \\ x_i + \phi, &\mbox{if } x_i < - \phi. }
$$	
It has recently been shown that ADMM converges linearly when applied to the
minimization of convex separable functions, under mild assumptions on the program input (see \cite{luo2012linear}), and, as such, Algorithm~\ref{alg: ADMM} can be expected to converge to the optimal solution of \eqref{eq: dks relaxation};
We stop Algorithm~\ref{alg: ADMM} when the primal and dual residuals 
$$
	\|X^{(\ell)} - W^{(\ell)}\|_F, \;
	 \|X^k - Z^{(\ell)}\|_F, \;
	  \|W^{(\ell+1)}- W^{(\ell)}\|_F, \;
	   \|Z^{(\ell+1)} - Z^{(\ell)} \|_F, \;
	    \|\Lambda_Q^{(\ell+1)} - \Lambda_Q^{(\ell)}\|_F
$$
are smaller than a desired error tolerance.

\begin{algorithm}[t!]
\caption{ADMM for solving \eqref{eq: dks relaxation}}
\label{alg: ADMM}
\algblockdefx[Step]{Step}{EndStep} %
    [2]{{\bf Step #1:} #2} %

\begin{algorithmic}
	\State {\bf Input:} $G=(V,E)$, $k \in \{1,\dots, N\}$, where $N=|V|$, and error tolerance $\epsilon$.\;
	\State {\bf Initialize:} $X^{(0)} = W^{(0)} = (k/N)^2\e\e^T$, $Y^{(0)} = -X$, $Q^{(0)} = \Lambda_Q^{(0)} =\Lambda_W^{(0)} = \Lambda_Z^{(0)} =0$.
	\For{$i=0, 1, \dots,$ until converged} 
		\Procedure{{\bf Step 1:} {Update $Q^{(\ell+1)}$}}
			\State $Q^{(\ell+1)} = P_{\tilde E} \rbra{X^{(\ell)} + Y^{(\ell)} - \Lambda_Q^{(\ell)} }.$
		\EndProcedure
		\Procedure{{\bf Step 2:} {Update $X^{(\ell+1)}$}}
			\State Let $\tilde X^{(\ell)} = Q^{(\ell+1)} + 2X^{(\ell)} - Z^{(\ell)} - W^{(\ell)} - \Lambda_W^{(\ell)}$.
			\State Take singular value decomposition $\tilde X^{(\ell)} = U \rbra{\Diag \x  }V^T$.
			\State Apply soft thresholding: 	$
				X^{(\ell+1)} = U \rbra{\Diag S_\tau(\x)} V^T
			$
		\EndProcedure
		\Procedure{{\bf Step {3}:} {Update $Y^{(\ell+1)}$}}
			\State $ Y^{(\ell+1)} = S_{\tau \gamma} \rbra{Y^{(\ell)} - \tau Q^{(\ell+1)} }$.
		\EndProcedure
		\Procedure{{\bf Step {4}:} {Update $W^{(\ell+1)}$}}
			\State Let $\tilde W^{(\ell)} = X^{(\ell+1)} - \Lambda_W^{(\ell)}$.
			\State Let $\beta_k = \rbra{ k^2 - \e^T \tilde W^{(\ell)} \e}/N^2$.
			\State Update $W^{(\ell+1)} = \tilde W^{(\ell)} + \beta_k \e\e^T$.
		\EndProcedure
		\Procedure{{\bf Step {5}: {Update $Z^{(\ell+1)}$}}}
			\State Let $\tilde Z^{(\ell)} = X^{(\ell+1)} - \Lambda_Z^{(\ell)}$.
			\State For each $i,j \in V$: $Z^{(\ell+1)}_{ij} = \min \{\max\{\tilde Z^{(\ell)}_{ij}, 0 \}, 1\}$
		\EndProcedure
		\Procedure{{\bf Step {6}:} {Update dual variables}}
			\State $\Lambda_Z^{(\ell+1)} = \Lambda_Z^{(\ell)} - \rbra{ X^{(\ell+1)} - Z^{(\ell+1)} }$.
			\State $\Lambda_W^{(\ell+1)} = \Lambda_W^{(\ell)} - \rbra{X^{(\ell+1)} - W^{(\ell+1)}}$.
			\State $\Lambda_P^{(\ell+1)} = P_{V\times V - \tilde E} \rbra{\Lambda_P^{(\ell)} - \rbra{X^{(\ell+1)} + Y^{(\ell+1)} } } $
		\EndProcedure
		\Procedure{{\bf Step {7}:} {Check convergence}}
			\State $r_p = \max \{\|X^{(\ell)} - W^{(\ell)}\|_F, \|X^k - Z^{(\ell)}\|_F\}$
			\State $r_d = \max\{  \|W^{(\ell+1)}- W^{(\ell)}\|_F, \|Z^{(\ell+1)} - Z^{(\ell)} \|_F, \|\Lambda_Q^{(\ell+1)} - \Lambda_Q^{(\ell)}\|\}$.
			\If{$\max\{r_p, r_d\} < \epsilon$}
				\State {\bf Stop:} algorithm converged.
			\EndIf	
		\EndProcedure
	\EndFor
\end{algorithmic}
\end{algorithm}

We evaluate the performance of our algorithm for a variety of random program inputs.
We generate random $N$-node graph $G$ constructed according to $(\rho_1)$ and $(\rho_2)$ for $q =0.25$ and various clique sizes $k \in (0, N)$  and edge addition probabilities $p \in [0, 1- q)$.
Each graph $G$ is represented by a random symmetric binary matrix $A$ with entries in the $(1:k) \times (1:k)$ block set equal to 1 with probability $1-q = 0.75$ independently and
remaining entries set independently equal to $1$ with probability $p$. 
For each graph $G$, we call Algorithm~\ref{alg: ADMM} to obtain solution $(X^*, Y^*)$;
regularization parameter $\gamma = 4/((1-p-q)k)$, augmented Lagrangian parameter
$\tau = 0.35$, and stopping tolerance $\epsilon = 10^{-4}$ is used in each call to Algorithm~\ref{alg: ADMM}.
We declare the planted dense $k$-subgraph to be recovered if $\|X^* - X_0\|_F/\|X_0\|_F < 10^{-3}$, where $X_0 = \bv\bv^T$ and $\bv$ is the characteristic vector of the planted $k$-subgraph.
The experiment was repeated 10 times for each value of $p$ and $k$ for $N = 250$ and $N=500$.
The empirical probability of recovery of the planted $k$-clique is plotted in Figure~\ref{fig: expts}. 
The observed performance of our heuristic closely matches that predicted by Theorem~\ref{thm: dks recovery}, with sharp transition to perfect recovery as $k$ increases past a threshold depending on $p$ and $N$.
However, our simulation results suggest that the constants governing exact recovery in Theorem~\ref{thm: dks recovery} may be overly conservative;
we have perfect recovery for smaller choices of $k$ than those predicted by Theorem~\ref{thm: dks recovery} for almost all choices of $p$.

\begin{figure}[t!]
	\caption{Simulation results for $N$-node graphs with planted dense $k$-subgraph
	Each entry gives the average number of recoveries of the planted subgraph per set of $10$ trials for the corresponding
	choice of $k$ and probability of adding noise edges $p$.
	Fixed probability of deleting clique edge $q=0.25$ was used in each trial. A higher rate of recovery is indicated by lighter
	colours. The graphs of the functions
	$f(p,q, N) = \sqrt{p N} \log N / ( {4(1-p-q)(1-p)} )$ and	
	$g(p,q, N) = ( {\sqrt{N} \log N}/({4(1-p-q)(1-p)^{3/2}} ))^{2/3},
	$
	are plotted as the solid and dashed lines, respectively, and
	approximate the theoretical thresholds for exact recovery given by \eqref{eqn: gap ass2} (with the estimate 
	of the scaling constant $c_1 \approx 1/4$); we should expect perfect recovery for all $k$ to the right of both curves.
	}
	\label{fig: expts}
	\centering
	\subfloat[{$N=250$}]{\includegraphics[width=0.5\textwidth]{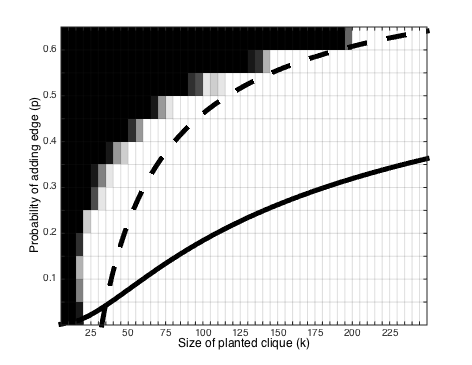} }  
	\subfloat[$N=500$]{\includegraphics[width=0.5\textwidth]{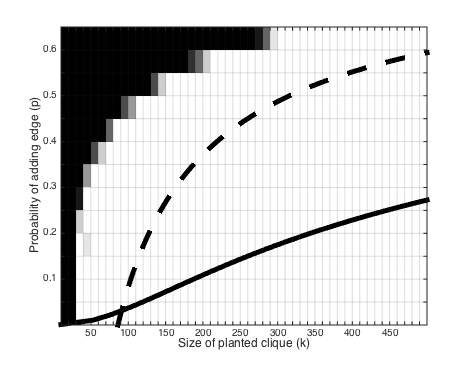}} 
\end{figure}

We repeated these experiments for bipartite graphs. 
Specifically, we generated random $(M,N)$-node bipartite graphs $G$ containing 
planted $(k_1, k_2)$-biclique with $q = 0.25$ for a variety of biclique sizes $(k_1, k_2)$ and $p$.
We call Algorithm~\ref{alg: ADMM} (with small modifications to address the lack of symmetry in $G$) to solve 
\eqref{eq: dkb relaxation} for each graph $G$; in particular, we set $X^{(0)} = W^{(0)} = k_1 k_2/(MN)$, $\beta_k = \big( (k_1 k_2) - \e^T \tilde W^{(\ell)} \e \big)/(MN)$,
and leave the rest of the algorithm unaltered.
We use the parameters $\gamma  = 4/((1-p-q)\sqrt{k_1 k_2})$, $\tau = 0.35$, and $\epsilon = 10^{-4}$
in each trial.
The obtained solutions were compared to the planted solutions as before to obtain the empirical probability of recovery
of the planted $(k_1,k_2)$-biclique over $10$ trials for each choice of $p$ and $(k_1,k_2)$. The experiment
was performed for the two graph sizes $(M,N) = (150, 225)$ and $(M,N) = (300, 450)$;
we choose $k_2 =  (3/2) k_1$, with $k_1 \in \{10, 20, \dots, 130, 140\}$ 
when $(M,N) = (150,225)$
and $k_1 \in \{20, 40, 60, \dots, 260, 280\}$ when $(M,N) = (300, 450)$.
As before, we observe a sharp transition to perfect recovery as $k_1$ increases past
some threshold depending on $p$, $M$, and $N$.
Again, it seems as though the predicted threshold may be overly conservative when compared to that observed empirically.

\begin{figure}[t!]
	\caption{Simulation results for $(M,N)$-node graphs with planted dense $(k_1, k_2)$-subgraph
	Each entry gives the average number of recoveries of the planted subgraph per set of $10$ trials for the corresponding
	choice of $k_1$ and probability of adding noise edges $p$.
	Fixed probability of deleting clique edge $q=0.25$ was used in each trial. 
	The graphs of the functions
	$f(p,q, N) = \sqrt{p N} \log N / ( {10(1-p-q)(1-p)} )$ and	
	$g(p,q, N) = ( {\sqrt{N} \log N}/({10(1-p-q)(1-p)^{3/2}} ))^{2/3},
	$
	are plotted as the solid and dashed lines, respectively, and
	approximate the theoretical thresholds for exact recovery given by \eqref{eqn: gap ass2} (with the estimate 
	of the scaling constant $c_1 \approx 1/10$); we should expect perfect recovery for all $k_1$ to the right of both curves.
	}
	\label{fig: expts}
	\centering
	\subfloat[{$(M,N)=(150, 225)$}]{\includegraphics[width=0.5\textwidth]{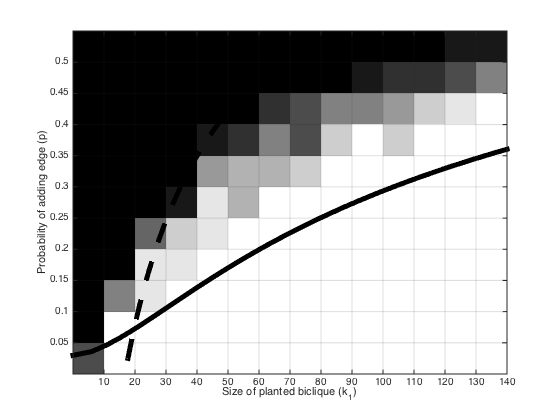} }  
	\subfloat[$(M,N)=(300, 450)$]{\includegraphics[width=0.5\textwidth]{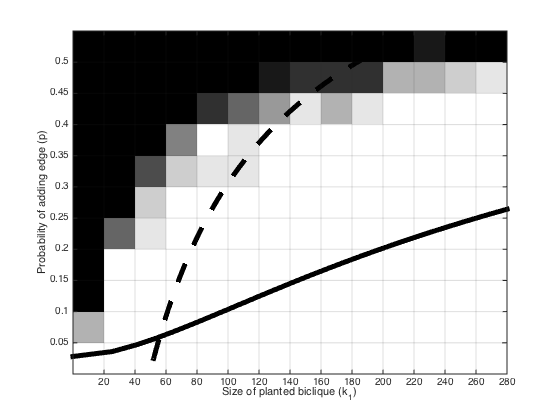}} 
\end{figure}

\section{Conclusions}
We have considered a convex optimization heuristic for identifying the densest $k$-node subgraph of a given graph,
with novel recovery properties. In particular, we have identified tradeoffs between the size and density of a planted subgraph ensuring that 
this subgraph can be recovered from the unique optimal solution of the convex program \eqref{eq: dks relaxation}.
Moreover, we establish analogous results for the identification of the densest  bipartite $(k_1, k_2)$-subgraph in a bipartite graph.
In each case, the relaxation relies on the decomposition of the adjacency matrices of candidate subgraphs as the sum of a  dense and sparse matrix, and is closely related to recent results regarding robust principal component analysis.

These results suggest several possible avenues for future research. First, although our recovery guarantees match those previously identified in the literature, these bounds may not be the best possible.
Rohe et al. \cite{rohe2012highest} recently established that  an $N$-node random graph sampled from the Stochastic Blockmodel can be partitioned into dense subgraphs of size $\Omega(\log^4 N)$ using a regularized maximum likelihood estimator.
It is unclear if such a bound can be attained for our relaxation.
It would also be interesting  to see if similar recovery guarantees exist for more general graph models; for example, can we find the largest planted clique in a graph with several planted cliques of varying sizes? 
Other potential areas of future research may also involve post-processing schemes for identifying the densest $k$-subgraph in the case that the optimal solution of our relaxation does not exactly correspond to the sum of
a low-rank and sparse matrix, and if a similar relaxation approach and analysis may lead to stronger recovery 
results for other intractable combinatorial problems, such
as the planted $k$-disjoint-clique \cite{ames2010convex} and clustering \cite{ames2012guaranteed} problems

\section{Acknowledgments}
This research was supported in part by the Institute for Mathematics and its Applications with funds provided by the National Science Foundation.
We are also grateful to Inderjit Dhillon, Stephen Vavasis, and Teng Zhang for their helpful comments and suggestions, and to Shiqian Ma for his insight and help implementing the ADMM algorithm used in Section~\ref{sec: expts}.


\appendix

\section{Appendices}

\subsection{Proof of Theorem~\ref{thm: KKT conds}}
\label{app: KKT proof}
	The convex program \eqref{eq: dks relaxation} admits a strictly feasible solution and, hence, \eqref{eq: dks relaxation} satisfies Slater's constraint qualification (see \cite[Equation~(3.2.7)]{borwein2006convex});
	for example, $X = (k^2/N^2) \e\e^T$ (with $Y = -X$)  satisfies the box constraints with strict inequality
	when $k < N$.
	Therefore, the Karush-Kuhn-Tucker conditions  applied to \eqref{eq: dks relaxation} state that a feasible solution $(X,Y)$ of \eqref{eq: dks relaxation}
	is optimal if and only if there exist multipliers $\lambda \in \R$, $H  \in \R^{V\times V}$, $M_1, M_2 \in \R^{V\times V}_+$, 
	and subgradients $\phi \in \partial \|X\|_*$, $\psi \in \partial\|Y\|_1$ such that
	\begin{align}
		\phi - \lambda \e\e^T + \sum_{(i,j) \in \tilde E} H_{ij} \e_i \e_j^T + M_1 - M_2 = 0 	\label{eq: orig DF X}\\
		\gamma \psi + \sum_{(i,j) \in \tilde E}  H_{ij} \e_i \e_j^T = 0  \label{eq: orig DF Y}\\
		[M_1]_{ij} (X_{ij} - 1) = 0 \;\;\forall\, i,j \in V \label{eq: orig CS} \\
		[M_2]_{ij} X_{ij} = 0 \;\; \forall \, i,j \in V \label{eq: orig nonneg CS}
	\end{align} 
	Taking $M_2 = 0$ ensures that \eqref{eq: orig nonneg CS} is satisfied for all $X$.
	Since $\bar X_{ij} = 1$ if $(i,j) \in \bar V \times \bar V$ and is $0$ otherwise, \eqref{eq: orig CS} is equivalent to \eqref{eq: KKT CS upperbound} when $X = \bar X$.	
	It is known (see, for example, {\cite[Section 3.4]{boyd2008subgradients}}) that
	$	\partial \|\bar Y\|_1 = \{ \bar Y + F: P_\Omega(F) = 0, \;\; \|F\|_\infty \le 1 \}. $
	Here, $P_{\Omega}$ is the projection onto the set of matrices with support contained in $\Omega$
	defined as in \eqref{eq: proj def}.
	We can substitute $\psi = \sign(\bar Y) + F$ in \eqref{eq: orig DF Y}
	for some matrix $F$ such that $P_\Omega (F) = 0$ and $\|F\|_\infty \le 1$.
	Moreover, since $\bar Y = 0$ for all $(i,j) \notin \tilde E$, \eqref{eq: orig DF Y} implies that $F_{ij} = 0$ for all $(i,j) \notin \tilde E$. Since the complement of
	$\tilde E$ is exactly $E \cup \{vv: v \in V\}$, this yields \eqref{eq: KKT CS edges}.	
	Similarly, the subdifferential of the nuclear norm at $\bar X$ is equal to the set
	$
		\partial \|\bar X\|_* = \{ \bar \bv \bar \bv^T / k + W : W \bar \bv = W^T \bar \bv = 0, \; \|W\| \le 1 \};
	$
	see \cite[Example 2]{watson1992characterization}.	
	Combining \eqref{eq: orig DF X} and \eqref{eq: orig DF Y} and substituting this formula for the subgradients of $\|\bar X\|_*$ into the resulting equation
	yields \eqref{eq: KKT dual feas} and \eqref{eq: KKT W def}.
	Thus, the conditions \eqref{eq: KKT dual feas}, \eqref{eq: KKT W def}, \eqref{eq: KKT F def}, \eqref{eq: KKT CS edges}, and \eqref{eq: KKT CS upperbound} are 
	exactly the Karush-Kuhn-Tucker conditions for \eqref{eq: dks relaxation} applied at $(\bar X, \bar Y)$, with the Lagrange multiplier $M_2$  taken to be $0$.
	
	We next show that $G(\bar V)$ has maximum density among all $k$-subgraphs of $G$ if $(\bar X, \bar Y)$ is optimal for \eqref{eq: dks relaxation}).
	Fix some subset of nodes  $\hat V \subseteq V$ of cardinality $k$.
	Let $\hat X = \hat \bv \hat\bv^T$ where $\hat\bv$ is the characteristic vector of $\hat V$ and
	let $\hat Y$ be the matrix constructed according to \eqref{eq: proposed Y} for $\hat V$.
	Note that both $\bar X$ and $\hat X$ are rank-one matrices with nonzero singular value equal to $k$.
	By the optimality of $(\bar X, \bar Y)$, we have
	$
		\|\hat X\|_* + \gamma \|\hat Y\|_1 = k + \gamma \|\hat Y\|_1 \ge k + \gamma \|\bar Y\|_1.
	$ 
	Consequently, $\|\hat Y\|_1 \ge \|\bar Y\|_1$ and $d(G(\hat V)) \le d(G(\bar V))$
	as required.
	
	It remains to show that the conditions $\|W\| <1$ and $\|F\|_\infty < 1$ imply that $(\bar X, \bar Y)$ is 
	the unique optimal solution of \eqref{eq: dks relaxation}.
	The relaxation \eqref{eq: dks relaxation} can be written as the semidefinite program
	\begin{equation}	\label{eq: dks sdp primal}
		\begin{array}{rl}
			\min	& \ds{\frac{1}{2} \bbra{\tr(R_1) + \tr(R_2)} + \gamma \inp{\e\e^T, Z}  }\\
			\st 	& \ds{ R := \mat{{cc} R_1 & X \\ X^T & R_2} \succeq 0 } \\
				& - Z_{ij} \le Y_{ij} \le Z_{ij}, \;\;\; \forall\, i,j \in V \\
				& \e^T X \e = k^2 \\
				& X_{ij} + Y_{ij} = 0 \;\;\; \forall\, (i,j) \in \tilde E \\
				&X_{ij} \le 1 \;\;\; \forall\, i,j \in V.
		\end{array}
	\end{equation}
	This problem is strictly feasible and, hence, strong duality holds.
	The dual of \eqref{eq: dks sdp primal} is	
	\begin{equation}	\label{eq: dks sdp dual}		
		\begin{array}{rl}		
			\max	& k^2 \lambda + \tr( \e\e^T M) \\
			\st 		& Q := \mat{{cc}	I & {-\lambda \e\e^T -  \sum_{(i,j) \in \tilde E} H_{ij} + M} \\
									{-\lambda \e\e^T -  \sum_{(i,j) \in \tilde E} H_{ji} + M}  & I }  \succeq 0\\
					& \begin{array}{ll} H_{ij} - S^1_{ij} + S^2_{ij} = 0 &\;\;\; \forall \, (i,j) \in \tilde E \\
					 S^1_{ij} - S^2_{ij} = 0 &\;\;\; \forall \, (i,j) \in (V\times V) - \tilde E \\
					 S^1_{ij} + S^2_{ij} = 0 &\;\;\; \forall \, i,j \in V.	\end{array} \\
					 & M, S^1 , S^2 \in \R^{N\times N}_{+},\; H\in \R^{N\times N}, \; \lambda \in \R, .
		\end{array}
	\end{equation}	
	Suppose that there exists multipliers $F, W, \lambda,$ and $M$ satisfying the hypothesis of Theorem~\ref{thm: KKT conds} such that $\|W\| < 1$ and $\|F\|_\infty$.
	Note that $\bar X = \bar R_1 = \bar R_2 =  \barX$, $\bar Y$ as constructed according to \eqref{eq: proposed Y}, and $\bar Z = \sign (\bar Y)$ defines a primal feasible solution for \eqref{eq: dks sdp primal}.
	We define a dual feasible solution as follows.
	If $(i,j)\in \Omega$ then $F_{ij} = 0$ and we take $\bar H_{ij} = -\gamma \bar Y_{ij} = \gamma$.	
	In this case, we choose $\bar S^1_{ij} = \gamma$ and $\bar S^2_{ij} = 0$.
	If $(i,j) \in \tilde E - \Omega$, we choose $\bar H_{ij} = -\gamma F_{ij}$ and take
	$ \bar S^1_{ij} = \gamma(1 - F_{ij})/2$, $\bar S^2_{ij} = \gamma (1 + F_{ij} ) /2$.
	Finally, if $(i,j) \notin \tilde E$, we take $\bar S^1_{ij} = \bar S^2_{ij} = \gamma/2$.
	Note that, since $|F_{ij} | < 1$ for all $i,j \in V$ and $\gamma > 0$, the entries of $\bar S^1$ are strictly positive and those of $\bar S^2$ are nonnegative
	with $\bar S^2_{ij} = 0$ if and only if $Y_{ij} - Z_{ij} < 0 $.
	Therefore, the dual solution $( \bar Q, \bar H, \bar S^1, \bar S^2)$ defined by the multipliers $F, W, \lambda, M$ is feasible and satisfies complementary slackness by construction.
	Thus, $(\bar R, \bar Y, \bar Z)$ is optimal for \eqref{eq: dks sdp dual} and $( \bar Q, \bar H, \bar S^1, \bar S^2)$ is optimal for the dual problem \eqref{eq: dks sdp dual}.
	
	We next establish that $(\bar X, \bar Y)$ is the unique solution of \eqref{eq: dks relaxation}.
	By \eqref{eq: KKT dual feas} and our choice of $\bar H$,
	$$
		\bar Q = \mat{{cc} I & - W - \bar X/k \\ - W^T - \bar X/k & I}.
	$$
	Note that $\bar R \bar Q = 0$ since $W \bar X =  W^T \bar X = 0$ and $\bar X^2/k = \bar X$.
	This implies that the column space of $\bar R$ is contained in the null space of $\bar Q$. Since $\bar R$ has rank equal to $1$, $\bar Q$ has rank at most $2N - 1$.
	Moreover, $W + \bar X/k $ has maximum singular value equal to $1$ with multiplicity $1$ since $\|W\| < 1$.
	Therefore, $\bar Q$ has exactly one zero singular value, since $\omega$ is an eigenvalue of $\bar Q-I$ if and only if $\omega$  or $-\omega$ is an eigenvalue
	of $W + \bar X/k$.
	Thus $\bar Q$ has rank equal to $2N-1$.
	
	To see that $(\bar X, \bar  Y)$ is  the unique optimal solution of \eqref{eq: dks relaxation}, suppose on the contrary that $(\hat R, \hat Y, \hat Z)$ is also optimal for \eqref{eq: dks sdp primal}.
	In this case, $(\hat X, \hat Y)$ is optimal for \eqref{eq: dks relaxation}.
	Since $( \bar Q, \bar H, \bar S^1, \bar S^2)$ is optimal for \eqref{eq: dks sdp dual}, we have $\hat R \bar Q = 0$ by complementary slackness.
	This implies that $\hat R = t \bar R$  and  $\hat X = t \bar X$ for some scalar $t \ge 0$ by the fact that the column and row spaces of $\hat R$ lie in the null space of $\bar Q$, which is
	spanned by $[\bar\bv; \bar\bv]$.
	Moreover, $\hat Y, \hat Z, \bar H, \bar S^1, \bar S^2$ also satisfy complementary slackness. 
	In particular, $\hat Y_{ij} = - \hat Z_{ij}$ for all $ij \in \Omega$ since $\bar S^1_{ij} \neq 0$.
	On the other hand, $\bar S^1_{ij} \neq 0$, $\bar S^2_{ij} \neq 0$ for all $(i,j) \notin \Omega$ and $\hat Y_{ij} = \hat Z_{ij} = 0$ in this case.
	It follows that $\supp(\hat Y) \subseteq \supp (\bar Y) = \Omega$ and
	$
		\hat Y = - P_{\Omega} \hat X = - t P_{\Omega} \bar X
	$
	by the fact that $P_{\Omega}( \hat X + \hat Y) = 0$.
	Finally, since $(\bar X, \bar Y)$ and $(\hat X, \hat Y)$ are both optimal for \eqref{eq: dks relaxation},
	$$
		\|\bar X\|_* + \gamma \|\bar Y\|_1 = \|\hat X\|_* + \gamma \|\hat Y\|_1 =
		t ( \|\bar X\|_* + \gamma \|\bar Y \|_1) .
	$$
	Therefore $t = 1$ and $(\bar X, \bar Y)$ is the unique optimal solution for \eqref{eq: dks relaxation}.
	\qed

\subsection{Proof of Lemma~\ref{lem: Tropp}}
\label{app: bernstein proof}
In this section, we establish the concentration bound on the norm of a mean zero matrix given by Lemma~\ref{lem: Tropp}.
To do so, we will show that
Lemma~\ref{lem: Tropp} is a special case of the following bound on the largest eigenvalue of a sum of random matrices.

\begin{theorem}[{\cite[Theorem~1.4]{tropp2011user}}]
	\label{thm: matrix bernstein}
	Let $\{X_k\} \in \Sigma^d$ be a sequence of independent, random, symmetric matrices of dimension $d$ satisfying
	$
		\E[X_k] =0 \mbox{ and } \|X\| \le R,
	$
	and let $S = \sum X_k$.
	Then, for all $t \ge 0$,
	\begin{equation}	\label{eqn: matrix bernstein}
		P \rbra{ \| S \| \ge t } \le d \cdot \exp \rbra{ -t^2/2}{\tilde\sigma^2 + Rt/3} \;\;\; \mbox{where } \tilde\sigma^2 := \left\|\sum_k \E(X_k^2) \right\|.
	\end{equation}
\end{theorem}		
	
To see that Lemma~\ref{lem: Tropp} follows as a corollary of 	Theorem~\ref{thm: matrix bernstein}, 
let $A \in \Sigma^n$ be a random symmetric matrix with i.i.d.~mean zero entries having variance $\sigma^2$ such that $|a_{ij}| \le B$ for all $i,j$.
Let  $\{X_{ij}\}_{1\le i\le j\le n} \in \Sigma^n$ be the sequence defined by
$$
	X_{ij} = \branchdef{ a_{ij} (\e_i\e_j^T + \e_j \e_i^T), &\mbox{if } i \neq j \\
					a_{ii} \e_i\e_i^T, &\mbox{if } i= j, }
$$					
where $\e_k$ is the $k$-th standard basis vector in $\R^n$.
Note that $A = \sum X_{ij}$.
It is easy to see that $\|X_{ij}\| \le |a_{ij}| \le B$ for all $1\le i\le j\le n$.
On the other hand, 
$$
	M = \sum_{1\le i\le j\le n} \E(X_{ij}^2) = \sum_{i=1}^n \rbra{\E(a_{ii}^2) \e_i\e_i^T + \sum_{j = i+1}^n \E(a_{ij}^2) (\e_i\e_i^T + \e_j \e_j^T) }= \sigma^2 n \cdot I,
$$	
by the independence of the entries of $A$.
Therefore, $ \tilde \sigma^2 = \|M\| = \sigma^2n$.
Substituting into \eqref{eqn: matrix bernstein} shows that
\begin{equation}
	P( \|A\| \ge t) \le n \exp \rbra{ -\frac{\sigma^2 n \log n / 2} {\sigma^2 n + Bt/3}  }
\end{equation}
for all $t \ge 0$.
To complete the proof,  we take $t = 6 \max\{\sigma \sqrt{n\log n}, B \log^2 n\}$ and consider the following cases.

First, suppose that $\sigma \sqrt{n \log n} \ge B \log^2 n$. In this case, we take $t = 6 \sigma \sqrt{n\log n}$.
Let $ f(t) = (t^2/2)/(\sigma^2 n + Bt/3)$. Then, we have
\begin{align*}
	f(t) = \frac{18 \sigma^2 n\log n}{\sigma^2 n  + 2 B \sigma \sqrt{n\log n} }
		\ge \frac{18 \sigma^2 n\log n}{\sigma^2 n + 2 \sigma^2 n \log n/\log^2 n}
		= \frac{18 \log n}{1 + 2/log n} \ge 9 \log n
\end{align*}		
if $n \ge 8$ by the assumption that $B \le \sigma \sqrt{n\log n}/\log^2 n$.
On the other hand, if $B \log ^2n > \sigma\sqrt{n\log n}$ we take $t = B \log^2 n$ and
\begin{align*}
	f(t) = \frac{18 B^2 \log^4 n}{\sigma^2 n + 2 B^2 \log^2n} > \frac{18 B^2 \log^4 n}{B^2\log^2 (\log n + 2) } > 9 \log n
\end{align*}
if $n \ge 8$.
In either case,	
$
	P(\|A\| \ge t ) \le \exp(-f(t)) \le n \exp(-9\log n) = n^{-8}.
$
\qed





\bibliographystyle{spmpsci} 
\bibliography{DKS-bib}
%



\end{document}